\documentclass[a4paper,12pt]{article}
 \usepackage[english]{babel}
  \usepackage{amsmath,amsfonts,amssymb,amsthm,bbm}
   \usepackage{graphics,epsfig,psfrag} 
   \usepackage{paralist,subfigure,float}

   \usepackage{hyperref} 

    \usepackage[active]{srcltx} 
    \usepackage{color,soul} 
    \usepackage[latin1]{inputenc}
    \usepackage[OT1]{fontenc}

\newtheorem{theorem}{Theorem}[section]
\newtheorem{corollary}[theorem]{Corollary}
\newtheorem{lemma}[theorem]{Lemma}
\newtheorem{proposition}[theorem]{Proposition}

\newtheorem*{meta}{Metatheorem}

\theoremstyle{definition}
\newtheorem{definition}[theorem]{Definition}
\newtheorem{remark}{Remark}

\DeclareMathOperator{\dist}{dist}
\DeclareMathOperator{\dv}{div}

\DeclareMathOperator\supp{supp}

\DeclareMathOperator\inter{int}

\def\N{\mathbb{N}}
\def\Z{\mathbb{Z}}

\def\R{\mathbb{R}}

\let\O=\Omega
\let\e=\varepsilon

\let\vt=\vartheta
\let\t=\tilde
\let\ol=\overline
\let\ul=\underline
\let\.=\cdot
\let\0=\emptyset

\let\mc=\mathcal
\let\di=\displaystyle

\def\1{\mathbbm{1}}
\def\pp{,\dots,}
\def\Sph{S^{N-1}}

\def\pe{principal eigenvalue}

\def\MP{maximum principle}
\def\SMP{strong maximum principle}

\newcommand{\su}[2]{\genfrac{}{}{0pt}{}{#1}{#2}}

\def\thm#1{Theorem~\ref{thm:#1}}
\def\seq#1{(#1_n)_{n\in\N}}
\def\limn{\lim_{n\to\infty}}
\def\as{\quad\text{as }\;}

\newenvironment{formula}[1]{\begin{equation}\label{#1}}
                       {\end{equation}\noindent}

\def\Fi#1{\begin{formula}{#1}}
\def\Ff{\end{formula}\noindent}

\def\PTF{pulsating travelling front}

\def\sub{asymptotic subset of spreading}
\def\super{asymptotic superset of spreading}
\def\ass{asymptotic speed of spreading}
\def\W{\mc{W}}

\title{\bf The Freidlin-G\"artner formula\\ for general reaction terms}
\author{Luca {\sc Rossi}
\thanks{CNRS, Ecole des Hautes Etudes en Sciences Sociales, PSL Research University,  
	Centre d'Analyse et Math\'ematiques Sociales, 190-198 avenue de France 
	F-75244 Paris Cedex 13, France}\\
}

\date{}

\begin{document}

\maketitle

\vspace{-15pt}

\begin{abstract}
We devise a new geometric approach to study the propagation of disturbance - compactly supported data - in reaction diffusion equations. The method builds a bridge between the propagation of disturbance and of
almost planar solutions.
It applies to very general reaction-diffusion equations.
The main consequences we derive in this paper are: 
a new proof of the classical Freidlin-G\"artner formula for the asymptotic speed of 
spreading for periodic Fisher-KPP equations, extension of the formula to 
the monostable, combustion and bistable cases, existence of the 
\ass\ for equations with almost periodic temporal dependence, 
derivation of multilevel propagation for
multistable equations.
%
\end{abstract}

\vspace{-10pt}



\section{Introduction}

We deal with the reaction-diffusion equation
\Fi{future}
\partial_t u=\dv(A(x)\nabla u)+q(x)\.\nabla u + f(x,u),\quad t>0,\ x\in\R^N.
\Ff
This type of equation models a huge variety of phenomena in biology, 
chemistry, physics and  social sciences, 
such as population dynamics, gene diffusion, combustion, flame propagation, 
spread of epidemics. In applications, one typically considers the Cauchy problem 
with compactly supported initial data. 
Assuming that $0$ and $1$ are two steady states, $1$ being attractive, a natural question is: at which speed
does the set where solutions are close to $1$ spread?
To formulate this question in a precise way one introduces the notion of the
{\em \ass}: for a given direction
$\xi\in\Sph$, this is a quantity $w(\xi)$ such that the solution $u$ to
\eqref{future} emerging from a compactly supported
initial datum $u_0\geq0,\not\equiv0$ satisfies
\Fi{c>w}
\forall c>w(\xi),\quad
u(t,x+ct\xi)\to0 \quad\text{as }t\to+\infty,
\Ff
\Fi{c<w}
\forall 0\leq c<w(\xi),\quad
u(t,x+ct\xi)\to1 \quad\text{as }t\to+\infty,
\Ff
locally uniformly in $x\in\R^N$. 

Even assuming that $u\to1$ locally uniformly as 
$t\to+\infty$, it is not obvious that such a quantity $w(\xi)$ exists and 
that it does not depend on the initial datum~$u_0$.
These~properties are proved by Aronson-Weinberger \cite{AW} in the case of
the equation $\partial_t u-\Delta u=f(u)$. There, of course, 
$w$ is independent of $\xi$.
Using large deviation probabilistic techniques, Freidlin and G\"artner
extend the result in \cite{FG,Freidlin} to equation \eqref{future} under the
assumption that $A,q,f$ are periodic in $x$ and that 
$0<f(x,u)\leq \partial_u f(x,0)u$ for $0<u<1$. The latter is known as Fisher-KPP
condition. The authors obtain the following formula for the speed of spreading
$$
w(\xi):=\min_{\su{z\in\R^N}{z\.\xi>0}}\frac{k(z)}{z\.\xi},
$$
where $k(z)$ is the periodic \pe\ of the linear operator
$$L_z:=\dv(A\nabla)-2z\.A\nabla+q\.z+
\big(-\dv(Az)-q\.z+z\.Az+\partial_u f(x,0)\big).$$
Several years later, in \cite{BHN} (see also \cite{W02}), it has been shown 
that, for given $e\in\Sph$, the quantity
$c^*(e):= \min_{\lambda>0}k(\lambda e)/\lambda$
coincides with the critical (or minimal) speed of \PTF s in the
direction~$e$ (see Section \ref{intro:periodic} for the definition). Therefore, 
Freidlin-G\"artner's formula can be rewritten as
\Fi{FG}
w(\xi)=\min_{e\.\xi>0}\frac{c^*(e)}{e\.\xi}.
\Ff
Namely, $w(\xi)$ is the
minimizer of the speed in the direction $\xi$ among all the fronts, even 
those in directions $e\neq\xi$.

Pulsating travelling fronts exist not only in  the Fisher-KPP case, but also for
other classes of reaction terms, though their critical speed no longer 
fulfils 
the
previous eigenvalue representation. Then one might wonder if the 
formula \eqref{FG} holds true beyond the Fisher-KPP case.
In the present paper, using a new PDE approach, we show that this is always the case,
whenever \PTF s are known to 
exist: monostable, combustion and bistable equations. We point out that in the latter two
cases, where $f(x,u)$ is nonpositive in a neighbourhood of $u=0$, $c^*(e)$ is the unique
speed for which a \PTF\ in the direction $e$ exists. 
\\


The spreading properties for heterogeneous -~in particular 
periodic~- reaction-diffusion equations have been widely studied in the 
literature, with other approaches than the probabilistic one of \cite{FG}. One 
is the viscosity solutions/singular perturbations method of 
Evans-Souganidis~\cite{ES} for the Fisher-KPP equation and 
Barles, Soner and Souganidis \cite{BSShom,BShom} for the bistable equation.
There the authors characterise the 
asymptotic propagation of solutions in terms of the evolution of a set governed by a 
Hamilton-Jacobi equation.
An abstract monotone system approach relying on a discrete time-steps 
formalism is used in Weinberger~\cite{W02}. It provides a general spreading
result for monostable, combustion and bistable periodic equations without using
the existence of
\PTF s. Then, in the monostable case, 
i.e., when~$f(x,\.)$ is assumed to be positive in $(0,1)$, the method 
itself allows the 
author to show the existence of pulsating fronts and to derive the 
Freidlin-G\"artner formula~\eqref{FG} in such case. Instead, in the combustion 
or bistable cases, \cite[Theorems~2.1 and~2.2]{W02} assert the existence of the 
spreading speed but do not relate it with the speeds of pulsating fronts.
Finally, a PDE approach is adopted 
by Berestycki-Hamel-Nadin~\cite{BHNadin}. 
This yields the Freidlin-G\"artner formula in the Fisher-KPP case and partially 
extend it to equations with general space-time dependent
coefficients. 
\\

Let us describe our method. Property \eqref{c>w} with $w(\xi)$ given by
\eqref{FG} is essentially a direct consequence of the comparison
principle between $u$ and  the critical \PTF s in all directions $e$ satisfying
$e\.\xi>0$. A bit of work is however required in order to handle initial data
which are not strictly less than $1$ and also because we aim to a
uniform version of \eqref{c>w}. The real novelty of this paper consists in the
derivation of \eqref{c<w}. The reason why this property is harder to obtain than
\eqref{c>w}
can be explained in the following way: a solution $u$ emerging from a compactly
supported initial datum has bounded upper level sets at any time, whereas
the upper level sets of a front contain a half-space.
This is why one can manage to bound~$u$ from above by a suitable
translation of any travelling front and eventually get \eqref{c>w}, but cannot
bound $u$ from below by a front in order to get \eqref{c<w}. 
Nevertheless, assuming that $u$ converges locally uniformly to $1$ as 
$t\to+\infty$, 
its upper level sets eventually contain arbitrarily large 
portions of half-spaces, and thus it will be possible to put 
%
some front below the limit of translations of $u$ by $\{(t_n,x_n)\}$
with $t_n\to+\infty$.
So, supposing by way of contradiction that \eqref{c<w} does not hold, the 
key is to find a sequence of translations of $u$, by a suitable $\{(t_n,x_n)\}$,
whose limit propagates with an {\em average speed} slower than a front.
%
%
%
%
%
% ...
To achieve this, we need to deal with all directions of spreading 
simultaneously, by considering the Wulff shape of the speeds. As a by-product, 
we obtain \eqref{c>w}-\eqref{c<w} uniformly with
respect to $(\xi,c)$.

Let us point out that, unlike in the singular perturbation approach,
we just consider translations of the original equations, without any rescaling. 
One of the advantages is that the equation we obtain in the limit keeps the same form as the original one, 
in particular the uniform ellipticity. 
Another difference with the singular perturbation approach is that,
roughly speaking, the latter 
makes use of the travelling fronts for the original equation in order to obtain 
the evolution equation for level sets in the limit, whereas 
our method works the other way around: we start with the analysis of
 the motion of the level sets and exploit the existence of 
 travelling fronts only at the end.\\

%

We now introduce the object we want to study.
\begin{definition}\label{def:W}
We say that a closed set $\W\subset\R^N$, coinciding with the closure of its
interior, is the {\em asymptotic set of spreading} for a reaction-diffusion 
equation if, for any bounded solution $u$ with a compactly supported
initial datum $0\leq u_0\leq1$ such that $u(t,x)\to1$ as $t\to+\infty$
locally uniformly in $x\in\R^N$, there holds
\Fi{W-}
\forall \text{ compact }K\subset\inter(\W),\quad
\inf_{x\in K}u(t,xt)\to1\quad\text{\;as 
}t\to+\infty,
\Ff
\Fi{W+}
\forall \text{ closed set $C$ such that }C\cap\W=\emptyset,\quad
\sup_{x\in C}u(t,xt)\to0\quad\text{\;as }t\to+\infty.
\Ff
If only \eqref{W-} (resp.~\eqref{W+}) holds we say that $\W$ is an
{\em asymptotic subset} (resp.~{\em superset}) {\em of spreading}.
\end{definition}

The above definition essentially says that the upper level sets of $u$
look approximately like $t\W$ for $t$ large. 
The requirement that $\W$ coincides with the closure of its interior 
automatically implies that the asymptotic set of 
spreading is unique when it exists. The objective of this paper is to derive 
the existence of the asymptotic set of spreading and to express it.

If the asymptotic set of spreading $\W$ is bounded and
star-shaped with respect to the origin - all properties that it is natural to 
expect -  we can write 
$$\W=\{r\xi\,:\,\xi\in S^{N-1},\ \ 0\leq r\leq w(\xi)\},$$ 
with $w$ upper semicontinuous. If $w$ is strictly positive and 
continuous then~$w(\xi)$ is
the \ass\ in the direction $\xi$, in the sense of \eqref{c>w}-\eqref{c<w}. In
addition, those limits 
hold uniformly with respect to $(\xi,c)\in\Sph\times\R_+$
such that $|c-w(\xi)|>\e$, for any $\e>0$.


\begin{remark}\label{rem:invasion}
The requirement in Definition \ref{def:W} that $u\to1$ as $t\to+\infty$ locally 
uniformly in space (or equivalently pointwise, due to parabolic
estimates and strong maximum principle) is automatically fulfilled by any 
$u_0\not\equiv0$ in the periodic case when~$f$ is of KPP type and 
$q$ is divergence-free with average $0$, see \cite{Freidlin}.
A sharp condition for possibly negative $f$ is used in \cite{BHRoques1}, later 
extended to the non-periodic setting in \cite{BHR}.
In the non-KPP cases, it may happen that solutions converge uniformly to 
$0$ (when the ``hair-trigger'' effect 
fails in the monostable case or when the solution is ``quenched'' in the 
combustion or bistable case, see \cite{AW}).
However, some sufficient conditions for the convergence to $1$ can be
readily 
obtained by comparison with solutions of homogeneous equations, to which the
classical results of \cite{AW} apply.

Let us also mention that the requirement that $u_0$ has
compact support (which is only needed for the \super\ property) 
can be relaxed by a suitably fast exponential decay, and that the restriction  
$u_0\leq1$ can be dropped if $f(x,s)<0$ for $s>1$.
\end{remark}

Actually, our approach applies to general space-time dependent equations 
provided that front-like solutions are available, yielding some upper and lower 
bounds on the \ass. Results of this type are derived in the work in progress 
\cite{BNmultiD} in 
the case of Fischer-KPP reaction terms, combining homogenization techniques 
with the tool of the generalized principle eigenvalue. It is not always 
possible to deduce the existence of 
the asymptotic speed of spreading from such bounds, and there are indeed 
cases where the speed of spreading does not exist (see \cite{GGN}). In the 
present paper, beside the periodic framework, we derive the existence of the \ass\ 
for combustion and bistable equations with almost periodic dependence in time,
which was not previously known.
One could wonder if some weaker compactness properties -~such as random stationary ergodicity~- may guarantee the existence
of the speed of spreading, as shown for the Fischer-KPP equation in dimension 
$1$ in \cite{FG,BNuni} and for advection equations in 
\cite{NX09}.
Also, problems set in domains with periodic holes, under Neumann
boundary condition, may also be envisioned. We have chosen not to treat these cases
in the present paper in order to avoid further technicalities.

To sum up, the bridge we build between the propagation of compactly supported data and almost planar solutions turns out to be a powerful
tool, providing:
\begin{itemize}
	\item A new proof of the Freidlin-G\"artner formula.
	\item Extension of the formula to monostable, combustion, bistable reaction terms.
	\item Multi-tiered propagation of disturbance for multistable equations.
	\item Control of the propagation of disturbance in general non-autonomous media 
	in terms of almost planar transition fronts, which yields 
	the existence of the \ass\ for almost periodic, time dependent equations.
	\item For very general autonomous or periodic equations, a strategy
	to reduce compactly supported data to
	{\em front-like}  data,	that is, satisfying
	\Fi{f-l}
	\lim_{x\.e\to-\infty}u_0(x)=1,\qquad
	u_0(x)=0\quad\text{for $x\.e$ large enough},
	\Ff
which allows us to derive the following generalised Freidlin-G\"artner's formula.
\end{itemize}
\begin{meta}
For equations which are periodic in space and time, 
the asymptotic set of spreading exists and is given by
$$\W=\{r\xi\,:\,\xi\in S^{N-1},\ \ 0\leq r\leq w(\xi)\},\qquad\
\text{with}\quad w(\xi)=\inf_{\su{e\in S^{N-1}}{e\.\xi>0}}\frac{c^*(e)}{e\.\xi},$$
where $c^*(e)$ is the speed of spreading for front-like data, i.e., satisfying \eqref{f-l}.
\end{meta}

We give the proof of this metatheorem in Section \ref{sec:meta} below, highlighting 
the hypotheses required on the operator - essentially the validity of the comparison principle and a priori estimates.
If the operator is autonomous and rotationally invariant, then in most cases (such as local parabolic operators)
the problem for a truly {\em planar} datum reduces to an equation in one single space variable. 
It then follows from the above result that $\W$ is a ball with radius independent of $N$,
showing that the propagation of disturbance does not depend on the dimension in which the problem is set, at least at the level of the average speed. We recall that, going beyond the average speed, 
the location of the interface of the disturbance does depend in general on the dimension (by a $\log t$ order in the autonomous case, see \cite{Gartner,Uchi-bi}). 

We have chosen to present the above statement in the vague form of metatheorem,
without specifying the hypotheses on the operator,
in order to give the flavour of the kind of results one can obtain with the method of this paper.
We believe it will be susceptible to application to a wide class of equations.
One application we present here is the extension
to higher dimension and to compactly supported data of the spreading result for multistable autonomous
equations derived in \cite{Terrace,Terracik} using the notion of {\em propagating terrace}. 

%
%

The paper is organized as follows: in Section \ref{intro:periodic} we state
the result about the asymptotic set of spreading in periodic media,
which yields 
Freidlin-G\"artner's formula for general reaction terms. In Section
\ref{intro:general} we present the extension to equations depending on both
space and time, without any periodicity assumption; we also state the new result in the
case of almost periodic temporal dependence.  The remaining sections are
dedicated to the proofs of these results. Namely, the asymptotic subset and
superset of spreading are dealt with in Sections \ref{sec:sub} and
\ref{sec:super} respectively; in both cases, we start with proving the 
most general
statements, from which we deduce the ones in periodic media.
Section \ref{sec:Shen} is dedicated to the derivation of the \ass\ for almost periodic time-dependent equations. 
The Metatheorem and the multi-tiered 
propagation for multistable equations are proved in Sections \ref{sec:meta}
and \ref{sec:terrace} respectively.


\subsection{Periodic case}\label{intro:periodic}

We say that a function defined on $\R^N$ is 
$1$-periodic if it is periodic in each direction of the canonical 
basis, with period $1$, i.e., if it 
is invariant under the translations
$x\mapsto x+z$ for $z\in\Z^N$.
We restrict to functions with period $1$ just 
for the sake of simplicity; what really matters is that all terms in the 
equation have the same period in any given direction of the basis.

Our hypotheses in the periodic case are the ones
 required to apply the results of \cite{pulsating,Xin91,Xin93} 
concerning the existence of \PTF s. The hypotheses intrinsic to our 
method are weaker (cf.~the next subsection).

The matrix field $A$ and the vector field $q$ are smooth
\footnote{More precisely, $A$ is $C^3$ and $q$ 
is $C^{1+\delta}$ in the monostable or combustion cases \cite{pulsating}, and 
$A,q$ are $C^\infty$ in the bistable case \cite{Xin91,Xin93}.}
and satisfy
\Fi{A}
A\text{ is symmetric, uniformly elliptic and $1$-periodic},
\Ff 
\Fi{q}
\dv q=0,\qquad
\int_{[0,1]^N}q=0,\qquad
q\text{ is $1$-periodic}.
\Ff 
The function $f:\R^N\times[0,1]\to\R$ is 
of class $C^{1+\delta}$, for some $\delta\in(0,1)$, and satisfies
\Fi{f}
\begin{cases}
\forall x\in\R^N,\quad f(x,0)=f(x,1)=0,\\
\exists S\in(0,1),\quad\forall x\in\R^N,\quad f(x,\.)\text{ is nonincreasing in
}[S,1],\\
\forall s\in(0,1),\quad f(\.,s)\text{ is $1$-periodic}.  
\end{cases}
\Ff
We further assume that $f$ is in one of the following three classes:
\begin{align}
\text{\em Monostable }\quad & \forall s\in(0,1),\quad 
\min_{x\in\R^N}f(x,s)\geq0,\quad
\max_{x\in\R^N}f(x,s)>0, \label{mono}\\ \notag \\
\text{\em Combustion } \quad &
\begin{cases}
\exists \theta\in(0,1),\quad\forall (x,s)\in\R^N\times[0,\theta],\quad
f(x,s)=0,\\
\di\forall s\in(\theta,1),\quad \min_{x\in\R^N}f(x,s)\geq0,\quad
\max_{x\in\R^N}f(x,s)>0,\\
\end{cases}\label{igni}\\ \notag \\
\text{\em Bistable }\quad &
f(x,s)=s(1-s)(s-\theta),\quad \theta\in(0,1/2).
\label{bi}
\end{align}
In the bistable case, in order to apply the results of Xin \cite{Xin91,Xin93}, 
we need the terms $A$, $q$ to be close to constants, in the 
following sense: 
\Fi{Xin}
\exists h>N+1,\quad \Big\|A-\int_{[0,1]^N}A\Big\|_{C^h([0,1]^N)}<k
,\quad \Big\|q-\int_{[0,1]^N}q\Big\|_{C^h([0,1]^N)}<k,
\Ff
where $k$ is 
a suitable quantity also depending on $h$.

Under the above hypotheses, it follows from \cite{pulsating} in the cases
\eqref{mono} or \eqref{igni}, and from \cite{Xin91,Xin93} in the case 
\eqref{bi}, that \eqref{future} admits {\em \PTF s} in any direction 
$e\in\Sph$. 
These are entire (i.e., for all times) solutions $v$ satisfying
\Fi{ptf}
\begin{cases}
 \forall z\in\Z^N,\ x\in\R^N,\quad v(t+\frac{z\.e}c,x)=v(t,x-z)\\
 v(t,x)\to1\ \text{ as }x\.e\to-\infty,\quad
 v(t,x)\to0\ \text{ as }x\.e\to+\infty,
\end{cases}
\Ff
for some quantity $c$, called {\em speed} of the front.
The above limits are understood locally uniformly in $t\in\R$.
In the monostable case \eqref{mono}, such fronts exist if and only if $c$ is
larger than 
or equal to a critical value, depending on $e$, that we call $c^*(e)$. In 
the other two cases they exist only for a single value of $c$, still denoted by 
$c^*(e)$.
We further know from \cite{pulsating,Xin91,Xin93} that, under the above 
hypotheses, $c^*(e)>0$ for all $e\in\Sph$, and any front $v(t,x)$ is increasing 
in $t$.


Here is the generalization of Freidlin-G\"artner's result.
\begin{theorem}\label{thm:Freidlin}
Under the assumptions \eqref{A}-\eqref{f} and either \eqref{mono} or
\eqref{igni} or \eqref{bi}-\eqref{Xin}, the set
\Fi{W}
\W:=\{r\xi\,:\,\xi\in S^{N-1},\ \ 0\leq r\leq w(\xi)\}, \quad\text{with }\
w(\xi):=\inf_{e\.\xi>0}\frac{c^*(e)}{e\.\xi},
\Ff
is the asymptotic set of spreading for \eqref{future}, in the sense of 
Definition \ref{def:W}. 

Moreover, $w$ is positive and continuous and thus
$w(\xi)$ is the \ass\ in the direction $\xi$.
\end{theorem}

The infimum in the definition of $w$ is actually a minimum because the 
function~$c^*$ is lower semicontinuous. To our knowledge, the semicontinuity of 
$c^*$ - Proposition \ref{pro:c*LSC} here - had not been previously derived. 
The weaker property $\inf c^*>0$ ensures in general that a function $w$ as in 
\eqref{W} is continuous, as shown in Proposition~\ref{pro:wC} below. 
The continuity of $w$ is crucial for our
method to work, and we emphasize that it does not require $c^*$ to be
continuous.

\begin{remark}
For $\xi\in\Sph$, let $e_\xi$ be a minimizer of the expression for $w(\xi)$ in 
\eqref{W}. There holds that
$$\forall\xi'\in\Sph,\quad w(\xi')\xi'\.e_\xi\leq c^*(e_\xi)=w(\xi)\xi\.e_\xi,$$
that is, $e_\xi$ is an exterior normal to $\W$ at the point $w(\xi)\xi$.
It then follows that, if $\W$ is smooth, the family
$(t\W)_{t>0}$ expands in the normal direction $\nu$ with speed $c^*(\nu)$, 
exactly as in the homogeneous case. The results of the next section show that, 
in a sense, this property holds true in very general contexts.
\end{remark}


\subsection{Extension to general heterogeneous media}\label{intro:general}

We will derive \thm{Freidlin} as a consequence of two results concerning 
equations with non-periodic space/time dependent coefficients, in the 
general form
\Fi{future-gen}
\partial_t u=\dv(A(t,x)\nabla u)+q(t,x)\.\nabla u+f(t,x,u),\quad t>0,\ x\in\R^N,
\Ff
under milder regularity hypotheses. 
We assume here that there is $\delta>0$ such that
\footnote{\ For us, $g\in C^{k+\delta}$, $k\in\N$, $\delta\in(0,1)$, means that the derivatives of $g$ of order $k$ are {\em uniformly} H\"older-continuous 
with exponent $\delta$; $g=g(t,x)$ 
is in $C^{k+\delta,h+\gamma}$ if $g(\.,x)\in C^{k+\delta}$ and 
$g(t,\.)\in C^{h+\gamma}$.}

\Fi{coeff-gen}
\begin{cases}
A\in C^{\delta,1+\delta}(\R^{N+1})\text{ is symmetric and uniformly 
elliptic},\\
q\in C^\delta(\R^{N+1}),\\
f\in W^{1,\infty}(\R\times\R^N\times[0,1]).
\end{cases}
\Ff 
Notice that the regularity of $A$ allows one to write the equation in 
non-divergence form and to apply Schauder's regularity theory.
Further hypotheses on $f$ are:
\Fi{f-gen}
\begin{cases}
\forall(t,x)\in\R\times\R^N,\quad f(t,x,0)=f(t,x,1)=0,\\
\exists S\in(0,1),\quad\forall (t,x)\in\R\times\R^N,\quad f(t,x,\.)\text{ is
nonincreasing in }[S,1],\\
\forall s\in(S,1),\quad \exists\; 
E \text{ relatively dense in }\R^{N+1},\quad
\inf_{(t,x)\in E}f(t,x,s)>0.   
\end{cases}
\Ff
We recall that a set $E$ is relatively dense in $\R^{N+1}$ if the function 
$\dist(\.,E)$ is bounded on $\R^{N+1}$. 
The above properties are fulfilled by 
all classes of reaction terms considered in the previous section;
the second one is needed for the sliding method to work, the last one prevents 
from having constant solutions between $S$ and $1$.
In the combustion~\eqref{igni} or bistable \eqref{bi} cases, the following 
condition is further satisfied:
\Fi{others}
\exists \theta\in(0,S],\quad \forall (t,x)\in\R\times\R^N,\quad 
f(t,x,\.)\text{ is nonincreasing in }[0,\theta].
\Ff
This is essentially the condition that yields the uniqueness of the speed of 
the fronts (cf.~Lemma \ref{lem:comparison} below). We extend $f(t,x,s)$ to $0$ 
for $s\notin[0,1]$.

Note that in the generality of the above hypotheses, it may happen 
that all solutions emerging from compactly supported initial data 
converge uniformly to $0$, as for instance for the equation 
$u_t-u_{xx}=u(1-u)(u-\theta)$ with $\theta>1/2$. Of course, in such case one 
cannot talk about spreading property, and our definition of asymptotic 
set of spreading is vacous. The analysis of conditions ensuring the 
contrary, i.e., eventual invasion for all or some initial data, is adressed 
in many papers (see the brief discussion in Remark \ref{rem:invasion}) and it is 
out of the scope of the present one.

\begin{definition}\label{def:gtf} An (almost planar) {\em transition front}
in the direction $e\in S^{N-1}$ connecting $S_2$ to $S_1$ is a bounded
solution $v$ for which there exists~$X:\R\to\R$ such~that
\Fi{gtf}
\begin{cases}
 v(t,x+X(t)e)\to S_2 & \text{as }x\.e\to-\infty\\
 v(t,x+X(t)e)\to S_1 & \text{as }x\.e\to+\infty\\
\end{cases}
\qquad\text{uniformly in $t$.}
\Ff
The quantities
$$\liminf_{t\to-\infty}\frac{X(t)}t,\qquad \limsup_{t\to+\infty}\frac{X(t)}t$$
are called respectively the {\em past speed} and the {\em future speed} of the 
transition front.
\end{definition}
Observe that, even if the function $X$ associated with a front 
is not unique, the past and future speeds are. It is readily seen that a 
\PTF\ with speed~$c$, i.e. satisfying \eqref{ptf}, fulfils the Definition 
\ref{def:gtf} of transition front with $X(t)=ct$, and thus has past and 
future 
speeds equal to $c$. The existence of almost planar 
transition fronts in non-periodic media is an open question, which is 
very interesting by itself. Owing to Theorems 
\ref{thm:intermediate} and \ref{thm:super} below, answering to this question 
in some particular cases will imply the spreading results for compactly 
supported initial data.

We will consider the family of {\em limiting equations} associated with 
\eqref{future-gen}:
\Fi{le}
\partial_t u=\dv(A^*(t,x)\nabla u)+q^*(t,x)\.\nabla u+f^*(t,x,u),
\Ff
where $A^*,q^*,f^*$ satisfy, for some sequence 
$(t_n,x_n)_{n\in\N}$ with $t_n\to+\infty$ as $n\to\infty$,
$$A(t+t_n,x+x_n)\!\to\! A^*\!(t,x),\quad q(t+t_n,x+x_n)\!\to\! q^*\!(t,x),\quad
f^*(t+t_n,x+x_n,s)\!\to\! f^*\!(t,x,s)$$
locally uniformly in $(t,x,s)\in\R\times\R^N\times\R$.
Roughly speaking, the family of limiting equations is the $\omega$-limit
set of the original equation \eqref{future-gen}. 

\begin{theorem}\label{thm:intermediate}
Assume that \eqref{coeff-gen}-\eqref{f-gen} hold.
Let $\ul c:S^{N-1}\to\R$ be such that
\Fi{infc>0}
\inf_{\Sph}\ul c>0,
\Ff
and, for all $e\in S^{N-1}$, $c<\ul c(e)$, $\eta<1$ and any limiting equation 
\eqref{le} on $\R_-\times\R^N$, there is a transition front
$v$ in the direction $e$, connecting $0$ and $1$ if $f$ satisfies
\eqref{others}, or connecting some $-\e<0$ and $1$ otherwise,  
which has past speed larger than $c$ and satisfies $v(0,0)>\eta$.
Then, the set $\W$ given by 
\Fi{Wsub_general}
\W:=\{r\xi\,:\,\xi\in S^{N-1},\ \ 0\leq r\leq w(\xi)\}, \quad\text{with }\
w(\xi):=\inf_{e\.\xi>0}\frac{\ul c(e)}{e\.\xi},
\Ff
is an asymptotic subset of spreading for \eqref{future-gen}.
\end{theorem}

Actually, \thm{intermediate} is in turn a
consequence of another result - \thm{subset} below - which provides a
general criterion for a given set to be an \sub. 
\thm{subset} 
is our most general
statement concerning the asymptotic subset of spreading, and it is the building
block of the whole paper. However, in the applications presented here -
spatial periodic case and temporal almost periodic case (see below) - the 
generality of the hypotheses of \thm{intermediate} is sufficient.

\begin{theorem}\label{thm:super}
Assume that \eqref{coeff-gen}-\eqref{f-gen} hold.
Let $\ol c:S^{N-1}\to\R$ be such that
\Fi{infc>>0}
\inf_{\Sph}\ol c>0,
\Ff
and, for all $e\in S^{N-1}$, $\eta<1$ 
and $R\in\R$, the equation 
\eqref{future-gen}  on $\R_+\times\R^N$ admits a transition front~$v$ in the 
direction 
$e$ connecting $0$ and $1$ with future speed less than or 
equal to $\ol c(e)$ satisfying
\Fi{v>}
\forall t\leq1,\ x\.e\leq R+\ol c(e)t,\qquad
v(t,x)\geq\eta.
\Ff
Then, the set $\W$ given by 
\Fi{W+gen}
\W:=\{r\xi\,:\,\xi\in S^{N-1},\ \ 0\leq r\leq w(\xi)\}, \quad\text{with }\
w(\xi):=\inf_{e\.\xi>0}\frac{\ol c(e)}{e\.\xi},
\Ff
is an asymptotic superset of spreading for \eqref{future-gen}.
\end{theorem}


If the functions $\ul c$ and $\ol c$ in Theorems \ref{thm:intermediate} and 
\ref{thm:super} coincide then one obtains the existence of the 
asymptotic set of spreading.
A typical application of \thm{intermediate} is with $\ul c(e)$ equal to the 
minimal speed among all transition fronts in the direction $e$ connecting $0$ 
and $1$ for any limiting equation. Instead, $\ol c(e)$ in \thm{super} should be 
the minimal speed only among the fronts for the original equation 
\eqref{future-gen}. In the periodic case considered in 
Section~\ref{intro:periodic}, the two quantities coincide because any limiting 
equation is simply a translation of the original one.
Moreover, in that case, we can restrict to \PTF s. This is how we derive 
\thm{Freidlin}.
However, for equations with arbitrary space-time dependence, it can happen that $\ul c<\ol c$ and that the 
asymptotic set of spreading does not exist, cf.~\cite{GGN}.

Another situation in which the above theorems entail the sharp propagation result
is when the equation does not depend on $x$ 
and therefore it may admit truly planar transition fronts. This is the case of the equation
\Fi{t-dep}
\partial_t u=\Delta u + f(t,u),\quad t>0,\ x\in\R^N.
\Ff
The existence of transition fronts for \eqref{t-dep} is derived by Shen under the assumption that 
$f(t,s)$ is {\em almost periodic} (a.p.~in the sequel) in $t$ uniformly in $s$, that is, $f$ is uniformly continuous and
from any sequence
$(t_n)_{n\in\N}$ in $\R$ one can extract a subsequence
$(t_{n_k})_{k\in\N}$ such that $f(\.+t_{n_k},\.)$ converges uniformly
in $\R^2$. 
The fronts are constructed in dimension $N=1$ and therefore they can be regarded as planar fronts in higher dimension.
Shen considers a reaction term of combustion type in \cite{Shen-comb-ap} and of bistable type in \cite{Shen-bi2} (the precise assumptions are given in Section \ref{sec:Shen} below). She proves that~\eqref{t-dep} admits a transition front, in 
the sense of Definition \ref{def:gtf}, of the form $v(t,x\.e)$, $e\in S^{N-1}$, such that
$X'(t)$ and $v(t,x\.e+X(t))$ are a.p.~in $t$ uniformly in~$x$.
Hence, being a.p., $X'$ satisfies the {\em uniform average} property:
\Fi{avspeed}
c^*:=\lim_{t\to\pm\infty}\frac1t\int_T^{T+t}X'(s)ds\quad
\text{exists uniformly in }T\in\R.
\Ff
In particular, $v$ has past and future speeds equal to $c^*$. 
\begin{corollary}\label{cor:Shen}
Assume that $f(t,s)$ is a.p.~in $t$ uniformly in $s$, uniformly Lipschitz-continuous and 
fulfils either the combustion or bistable condition (see~Section \ref{sec:Shen}).
Then $\W:=\ol B_{c^*}$ is the asymptotic set of spreading for \eqref{t-dep} in any dimension~$N$.
\end{corollary}
The question about the spreading speed for \eqref{t-dep} when $f$ is monostable remains open;
we cannot apply Theorems \ref{thm:intermediate}, 
\ref{thm:super} because transition fronts are not
known to exist in such case. This is only known under the stronger
Fisher-KPP condition, see \cite{Shen-mono,NR1}. 
However, the result of Corollary \ref{cor:Shen} has already been obtained in such case
in \cite[Proposition 2.6]{NR1} using a different argument from the one presented here.


\subsection{Multistable equations}\label{intro:multi}

Here we consider the autonomous equation
\Fi{homo}
\partial_t u=\Delta u + f(u),\quad t>0,\ x\in\R^N.
\Ff
We make no assumptions on the number of steady states between $0$ and $1$, nor on their stability.
We just require the following minimal hypotheses:
\Fi{hyp-homo}
	f\in C^1(\R),\qquad f(0)=f(1)=0,
\Ff
\Fi{DGM}
\begin{cases}
\text{There is a solution $\ul u$ of \eqref{homo} in dimension\!\;1 with a compactly supported,}\\
\text{continuous initial datum $0\leq \ul u_0<1$ such that $\ul u(t,x)\to1$ as $t\to+\infty$.}
\end{cases}
\Ff

\begin{theorem}\label{thm:terrace}
Under the assumptions \eqref{hyp-homo}-\eqref{DGM},
there exist some numbers $M\in\N$, $0=\theta_0<\cdots<\theta_M=1$ and $c_1>\cdots> c_M>0$
such that any bounded solution $u$ with a compactly supported
initial datum $0\leq u_0\leq1$ such that $u(t,x)\to1$ as $t\to+\infty$
locally uniformly in $x\in\R^N$ satisfies
\Fi{c>cm}
\forall c>c_m,\quad
\limsup_{t\to+\infty}\bigg(\sup_{|x|\geq  ct}u(t,x)\bigg)\leq \theta_{m-1},
\Ff
\Fi{c<cm}
\forall c<c_m,\quad
\liminf_{t\to+\infty}\bigg(\inf_{|x|\leq  ct}u(t,x)\bigg)\geq \theta_m.
\Ff
\end{theorem}
It follows in particular that, as $t\to+\infty$,
$$
\forall m\in\{1,\dots,M-1\},\ c_{m+1}<c<c'<c_m,\quad
\sup_{ct\leq |x|\leq  c't}|u(t,x)-\theta_m|\to0,
$$
$$
\forall c'<c_M<c_1<c,\quad
\inf_{|x|\leq  c't}u(t,x)\to1,\quad 
\sup_{|x|\geq  ct}u(t,x)\to0.
$$
Namely, $u$ has the following multi-tiered cake shape
far from the regions~$|x|\sim c_m t$ for large $t$:
$$u(t,x)\sim \sum_{m=1}^{M+1}\theta_m\1_{B_{c_mt}\setminus 
B_{c_{m-1}t}}(x)+\theta_M\1_{B_{c_Mt}}(x).$$
The speeds $c_1,\dots,c_M$ are provided by the {\em propagating terraces}
derived in \cite{Terrace} in dimension 1. 
The restriction on the dimension is intrinsic to the method used in \cite{Terrace}, which
relies on the zero-number principle.
We point out that \thm{terrace} is new even in
dimension $1$, because no results have been previously obtained for compactly supported initial data. 
Let us mention that it is also proved in \cite{Terrace,Terracik},
always in dimension 1, that in each region $|x|\sim c_m t$
the solution develops one or more interfaces approaching a planar wave. 
The extension of this precise convergence result to higher dimension is the object 
of the work in progress~\cite{TerraceN}.


\section{Subset of spreading}\label{sec:sub}


\subsection{The general sufficient condition}

In this subsection we derive a sufficient condition for a compact set 
$\W\subset\R^N$, which is
star-shaped with respect to the origin, to be an asymptotic subset of spreading 
for~\eqref{future-gen}. A set of this type can be expressed by
\Fi{star}
\mc{W}=\{r\xi\,:\,\xi\in S^{N-1},\ \ 0\leq r\leq w(\xi)\},\quad
\text{with }\ w\geq0\ \text{ upper semicontinuous}.
\Ff
We will assume that $\W$ fulfils the {\em uniform interior ball 
condition}, that is, that there exists $\rho>0$ such that for all $\hat 
x\in\partial\W$, there is $y\in\W$ satisfying
$$|y-\hat x|=\rho,\qquad \ol B_\rho(y)\subset\W.$$
We say that $\nu(\hat x):=(\hat x-y)/\rho$ is
an exterior unit normal at $\hat x$ (possibly not unique).

We will need two auxiliary results.
\begin{lemma}\label{lem:invasion}
Under the assumptions \eqref{coeff-gen}-\eqref{f-gen}, let $u\in 
C^{1+\delta/2,2+\delta}(\R_-\times\R^N)$ be a supersolution of the equation
\Fi{past-gen}
\partial_t u-\dv(A(t,x)\nabla u)+q(t,x)\.\nabla u=f(t,x,u),\quad t<0,\ x\in\R^N,
\Ff
for which there is $H\subset \R^N$ such that
$$\sup_{x\in H}\,\dist(x,\R^N\setminus H)=+\infty,\qquad \inf_{t<0,\;x\in 
H}u(t,x)>S,$$
where $S$ is from \eqref{f-gen}. Then, 
$$\liminf_{\dist(x,\R^N\setminus H)\to+\infty}\left(\inf_{t<0}u(t,x)\right)
\geq1.$$
\end{lemma}

\begin{proof}
Assume by contradiction that
$$h:=\liminf_{\dist(x,\R^N\setminus H)\to+\infty}\left(\inf_{t<0}u(t,x)\right)
\in(S,1).$$
Let $(t_n)_{n\in\N}$ in $\R_-$ and $(x_n)_{n\in\R}$ in $H$ be such that
$$\dist(x_n,\R^N\setminus H)\to+\infty\quad\text{and}\quad u(t_n,x_n)\to 
h\as n\to\infty.$$
The functions $u(\.+t_n,\.+x_n)$ 
converge as $n\to\infty$ (up to subsequences) locally uniformly on 
$\R_-\times\R^N$ to a supersolution $u_\infty$ of a limiting equation 
\eqref{le}. Furthermore, 
$$u_\infty(0,0)=\min_{\R_-\times\R^N} u_\infty=h\in(S,1).$$
Notice that $f(t,x,s)\geq0$ if $s\in[S,1]$ by the first two 
conditions in \eqref{f-gen}, and then the same is 
true for $f^*$. It then follows from the parabolic strong \MP\ that 
$u_\infty=h$ in $\R_-\times\R^N$, whence $f^*(t,x,h)=0$ for 
$t\leq0$, $x\in\R^N$. Let us check that also the last property of~\eqref{f-gen} 
is inherited by $f^*$. Fix $s\in(S,1)$ and let~$E$ be the relatively dense 
set 
in $\R^{N+1}$ 
on which $f(\.,\.,s)$ has positive infimum. The fact that $E$ is relatively 
dense means that there is a compact set 
$K\subset\R^{N+1}$ such that $E\cap(K+\{(\tau,\xi)\})\neq\emptyset$, for any 
$(\tau,\xi)\in\R\times\R^N$.
Hence, for any $(\tau,\xi)\in\R\times\R^N$,
$$\max_{(t,x)\in(K+\{(\tau,\xi)\})}f^*(t,x,s)=\limn
\max_{(t,x)\in(K+\{(\tau+t_n,\xi+x_n)\})}f(t,x,s)\geq\inf_{(t,x)\in 
E}f(t,x,s)>0,$$
that is, $f^*$ fulfils the last condition in~\eqref{f-gen}. This is impossible 
because $f^*(t,x,h)=0$ for $t\leq0$, $x\in\R^N$.
\end{proof}

The second auxiliary lemma is a comparison principle. The proof relies on a 
rather standard application of the sliding method, in 
the spirit of \cite{BN92,pulsating}, and it is presented here in the appendix.
\begin{lemma}\label{lem:comparison}
Assume that \eqref{coeff-gen}-\eqref{f-gen} hold.
Let $\ul u,\ol u\in C^{1+\delta/2,2+\delta}(\R_-\times\R^N)$ be respectively a
sub and a supersolution of \eqref{past-gen} satisfying, for some $e\in S^{N-1}$,
\Fi{olu>}
\ol u>0,\qquad
\liminf_{x\.e\to-\infty}\ol u(t,x)\geq1\quad\text{uniformly in }\;t\leq0,
\Ff
$\ul u\leq1$ and there exists $\gamma>0$ such that either
\Fi{ulu<others}
\forall s>0,\ \exists L\in\R,\quad\ul u(t,x)\leq s\quad\text{for }\ t\leq0,\ 
x\.e\geq \gamma t+L
\Ff
if $f$ satisfies \eqref{others}, or
\Fi{ulu<mono}
\exists L\in\R,\quad\ul u(t,x)\leq0\quad\text{for }\ t\leq0,\ 
x\.e\geq \gamma t+L
\Ff
otherwise. Then, $\ul u(t,x)\leq\ol u(t,x)$ for 
$(t,x)\in\R_-\times\R^N$.
\end{lemma}

We are now in the position to derive our key result.

\begin{theorem}\label{thm:subset}
Under the assumptions \eqref{coeff-gen}-\eqref{f-gen}, let $\W\subset\R^N$ be a 
compact set, star-shaped with respect to the origin 
and satisfying the uniform interior ball condition. 
Suppose that for all $\eta,k<1$, $\hat x\in\partial\W$ and exterior unit normal 
$\nu(\hat x)$ at $\hat x$, 
every limiting equation \eqref{le}  on $\R_-\times\R^N$ admits a 
subsolution $v\in C^{1+\delta/2,2+\delta}(\R_-\times\R^N)$ satisfying 
$v\leq1$, $v(0,0)>\eta$, and, for some $c>k \hat 
x\.\nu(\hat x)$, either
\Fi{v<others}
\forall s>0,\ \exists L\in\R,\quad v(t,x)\leq s\quad\text{if }\ t\leq0,\ 
x\.\nu(\hat x)\geq ct+L
\Ff
if $f$ satisfies \eqref{others}, or
\Fi{v<mono}
\exists L\in\R,\quad
v(t,x)\leq0\quad\text{if }\ t\leq0,\
x\.\nu(\hat x)\geq ct+L
\Ff
otherwise. Then $\W$ is an asymptotic subset of spreading for~\eqref{future-gen}.
\end{theorem}

\begin{proof}
First, the interior ball condition implies that $\W$ coincides with the closure 
of its interior.
Let $u$ be as in Definition \ref{def:W}. Notice that $u\leq1$ by the
comparison principle. Fix $\eta\in(0,1)$ and $t>0$. We start dilating $\W$ 
unitl it touches the level set $\{u(\.,t)=\eta\}$. Namely, we define
$$\mc{R}^\eta(t):=\sup\{r\geq0\,:\,\forall x\in r\W,\ u(t,x)>\eta\}.$$
On one hand, the above quantity is finite because it is well known that 
$u(t,x)\to0$ as $|x|\to\infty$ (with Gaussian decay, see 
e.g.~\cite{Friedman}),
on the other,
$\mc{R}^\eta(t)\to+\infty$ as $t\to+\infty$ because $u(t,x)\to1$ as 
$t\to+\infty$ locally uniformly in $x\in\R^N$.
In order to prove that $\W$ satisfies the condition \eqref{W-} of the 
asymptotic subsets of spreading it is sufficient to show that
$$
\forall\eta\in(0,1),\quad
\liminf_{t\to+\infty}\frac{\mc{R}^\eta(t)}t\geq1.
$$
Indeed, the above condition implies that, for all $\eta,\e\in(0,1)$, 
$u(t,xt)>\eta$ for $x\in(1-\e)\W$ and $t$ larger than some $t_{\eta,\e}$. 
Then, for any compact $K\subset\inter(\W)$, $\e$ can be chosen in such a way 
that 
$(1-\e)^{-1}K\subset\W$, that is, $K\subset(1-\e)\W$. It follows that, for any 
$\eta<1$, $\inf_{x\in K}u(t,xt)>\eta$ if $t>t_{\eta,\e}$, which is precisely 
condition \eqref{W-}.

Assume by way of contradiction that there exist $\eta,k\in(0,1)$ such that 
\Fi{R<kt}
\liminf_{t\to+\infty}\frac{\mc{R}^\eta(t)}t<k.
\Ff
Clearly, \eqref{R<kt} still holds if one increases $\eta$. Then, 
we can assume without loss of generality that $\eta\in(S,1)$, where $S$ is from 
\eqref{f-gen}. Let us drop for simplicity the $\eta$ in the notation 
$\mc{R}^\eta$. 
We have that $\liminf_{t\to+\infty}(\mc{R}(t)-kt)=-\infty$. We set
$$\forall n\in\N,\qquad t_n:=\inf\{t\geq0\,:\,\mc{R}(t)-kt\leq-n\}.$$
It follows from the continuity of $u$ that the function 
$\mc{R}$ is lower semicontinuous. We then deduce that the above infimum is a 
minimum, i.e.~$\mc{R}(t_n)-kt_n\leq-n<\mc{R}t-kt$ for all $0\leq t<t_n$, and 
that
$t_n\to+\infty$ as $n\to\infty$. To sum up, there holds
\Fi{R<}
\lim_{n\to\infty} t_n=+\infty,\qquad
\forall n\in\N,\ t\in[0,t_n),\quad \mc{R}(t_n)-k(t_n-t)<\mc{R}(t).
\Ff
Take $n\in\N$. By the definition of $\mc{R}$, there exists 
$x_n\in\partial(\mc{R}(t_n)\W)$ such that
$u(t_n,x_n)=\eta$. We know that $|x_n|\to\infty$ as $n\to\infty$.
Define the sequence of functions $(u_n)_{n\in\N}$ by 
$u_n(t,x):=u(t+t_n,x+x_n)$. These functions are equibounded in $C^{1,2}_{loc}$ 
by standard parabolic interior estimates (see, e.g.,~\cite{Friedman}), and 
therefore they converge (up to subsequences) locally uniformly to a solution 
$u^*$ 
of some limiting equation~\eqref{le} which satisfies $u^*(0,0)=\eta$. The 
strong maximum principle then yields~$u^*>0$.

Take $T\in[0,t_n]$ and $x\in(\mc{R}(t_n)-kT)\W$.
It follows from \eqref{R<} that $x\in \mc{R}(t_n-T)\W$, whence
$u(t_n-T,x)\geq\eta$. We then derive
\Fi{un>}
\forall T\in[0,t_n],\ x\in(\mc{R}(t_n)-kT)\W-\{x_n\},\quad
u_n(-T,x)\geq\eta.
\Ff
For $n\in\N$, call 
$$\hat x_n:=\frac{x_n}{\mc{R}(t_n)}\in\partial\W,\qquad y_n:=\hat 
x_n-\rho\nu(\hat x_n),$$ 
where $\rho$ is the radius of the uniform interior ball condition and $\nu(\hat 
x_n)$ is an associated exterior unit normal at $\hat x_n$ (and thus $y_n$ is 
the centre of the ball). The situation is depicted in Figure \ref{fig:W}.
\begin{figure}[H]
\hspace{-30pt}\vspace{-10pt}
\includegraphics[height=8cm]{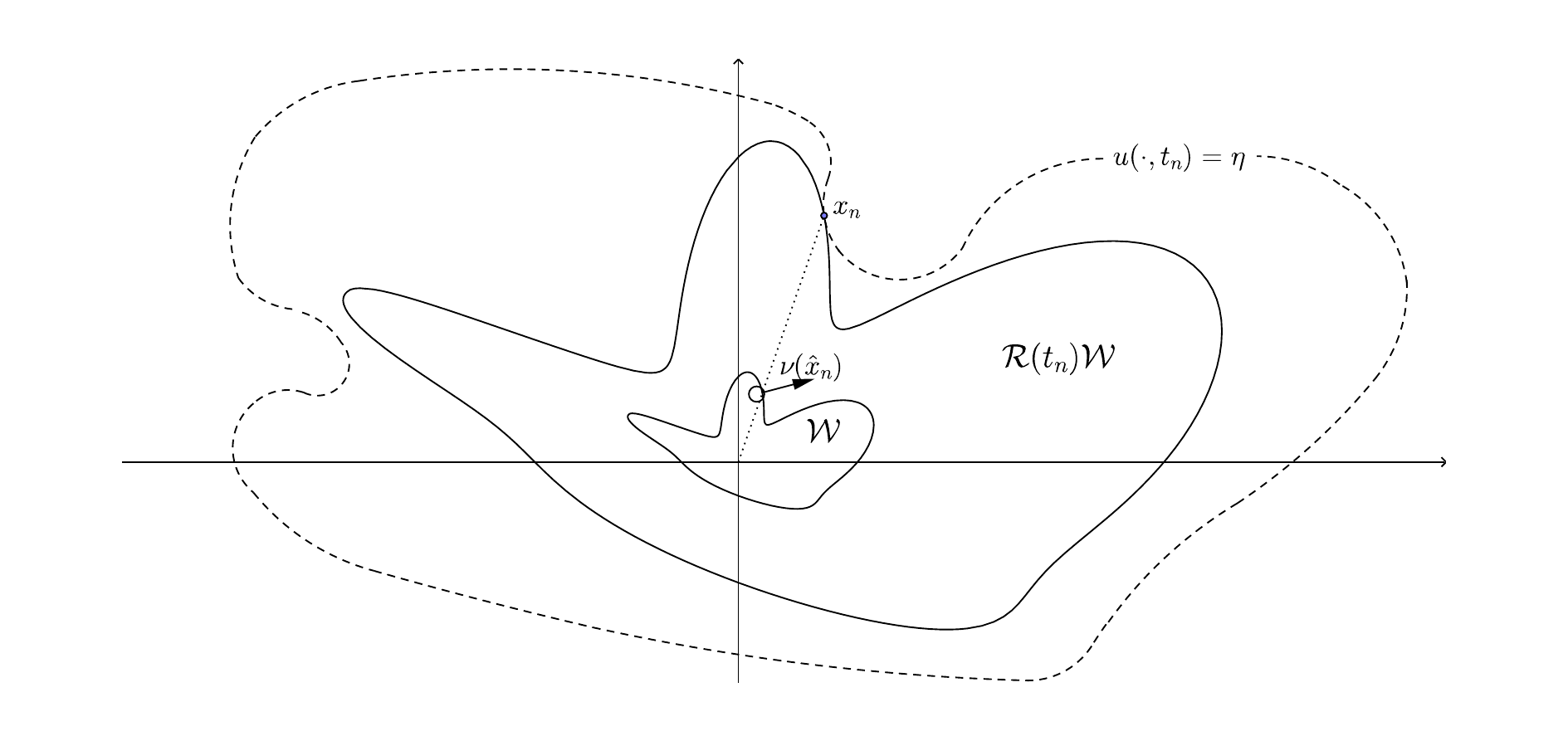}
\caption{Dialation of $\mc{W}$ until touching the level set 
$\{u(\.,t_n)=\eta\}$, at the point $x_n$.}
\label{fig:W}
\end{figure}

Let $\hat x$, 
$\nu(\hat x)$ be the limits of (subsequences of) $(\hat x_n)_{n\in\N}$, 
$(\nu(\hat x_n))_{n\in\N}$. Because $\W$ is closed, $\nu(\hat x)$ is an 
exterior unit normal at $\hat x\in\partial\W$.
We claim that, for any $T\geq0$, as $n\to\infty$, the sets 
$(\mc{R}(t_n)-kT)\W-\{x_n\}$ invade the half-space 
$$\mc{H}_T:=\{x\in\R^N\,:\,x\.\nu(\hat x)<-k(\hat x\.\nu(\hat x))T\},$$
in the sense that
\Fi{half-space}
\mc{H}_T\subset\bigcup_{M\in\N}\,\bigcap_{n\geq M}
\big((\mc{R}(t_n)-kT)\W-\{x_n\}\big)
\Ff
(see Figure \ref{fig:H}).
This property is a consequence of the fact that these sets 
satisfy the interior ball condition with radii 
$(\mc{R}(t_n)-kT)\rho$, which tends to $\infty$ as $n\to\infty$.
\begin{figure}[H]
 \centering
 \subfigure
   {\includegraphics[height=4cm]{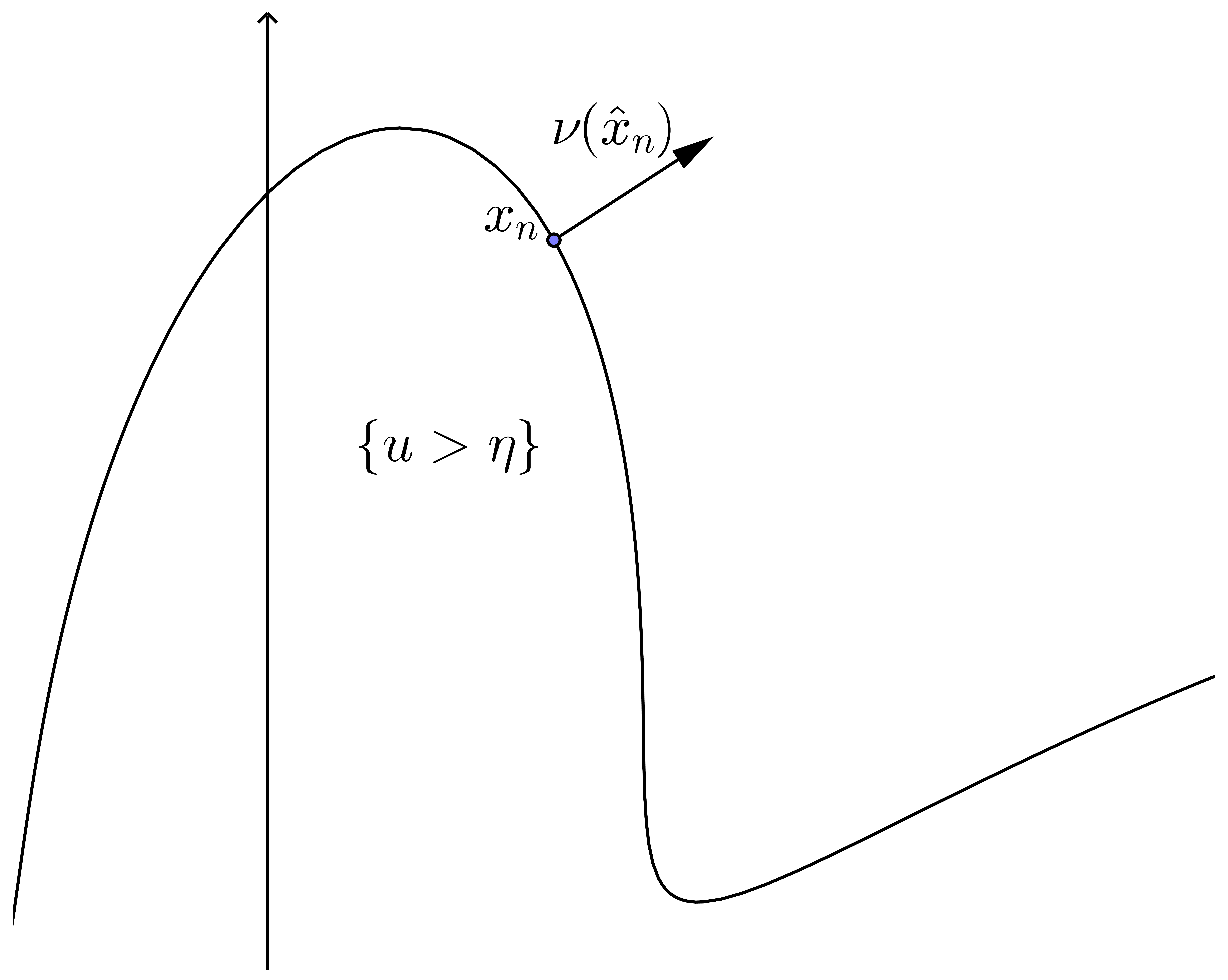}}
 \hspace{10mm}
 \subfigure
   {\includegraphics[height=4cm]{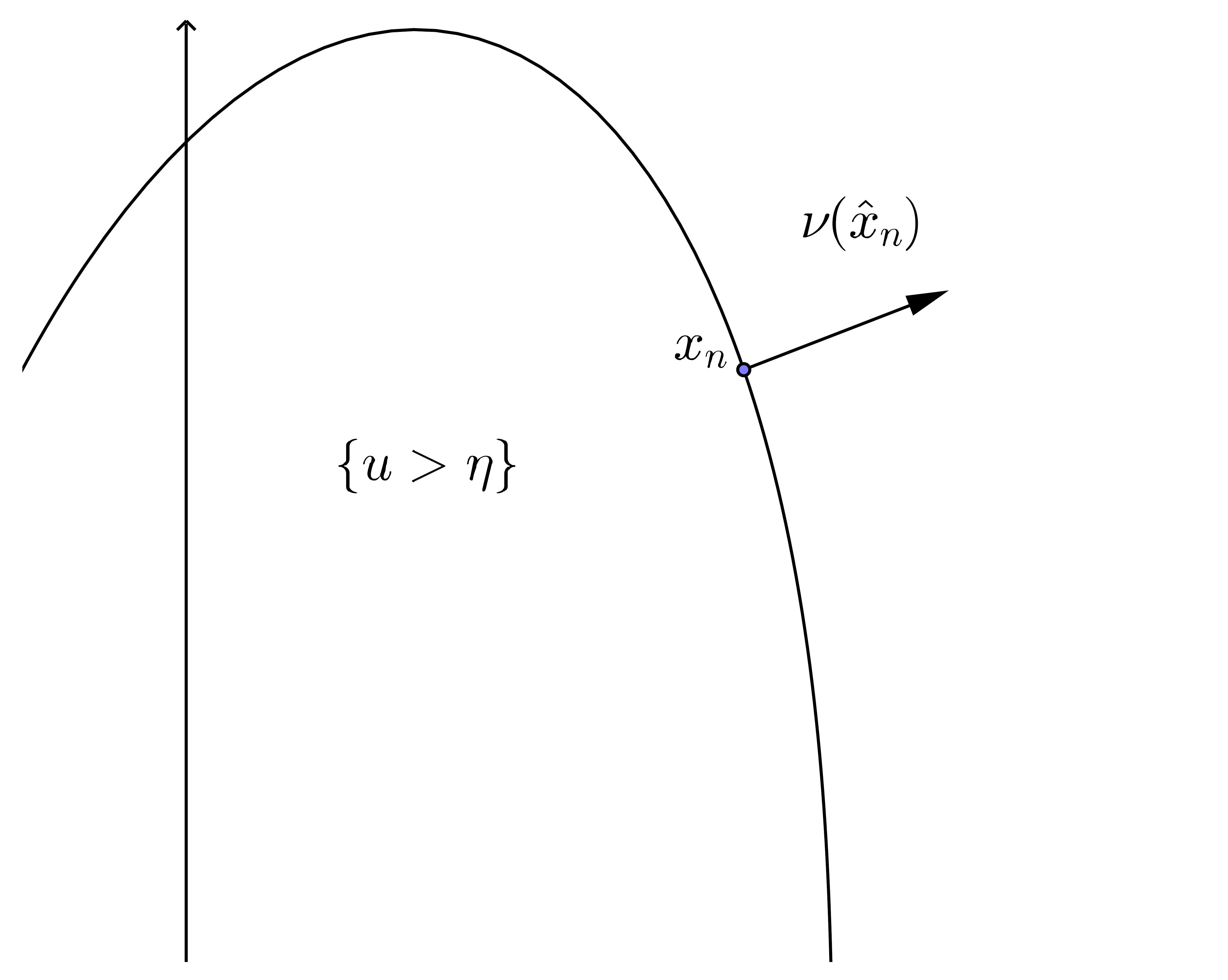}}
  \hspace{-7mm}
 \subfigure
   {\includegraphics[height=4cm]{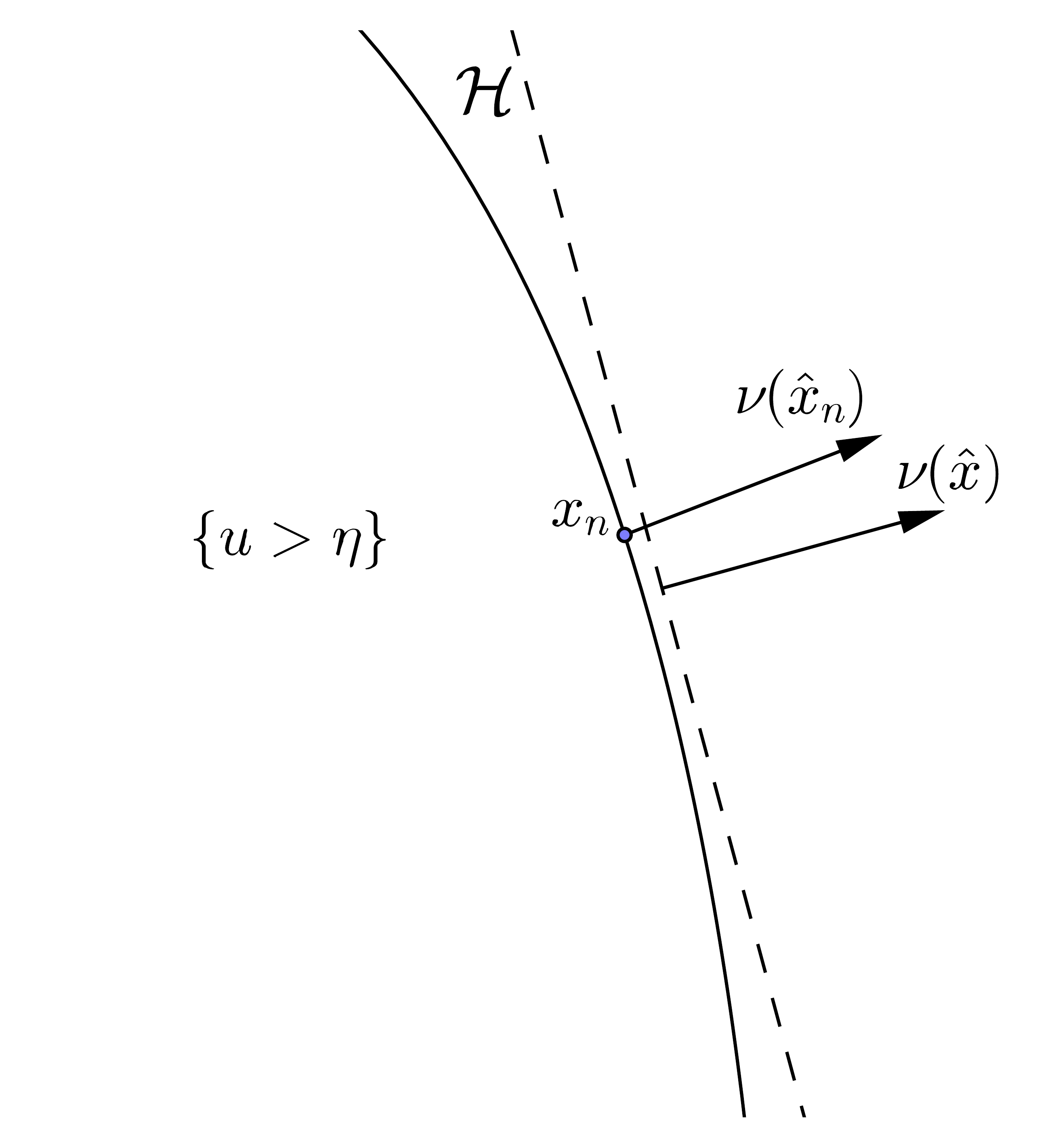}}
 \caption{Invasion of the half-space $\mc{H}$ by $\mc{R}(t_n)\W-\{x_n\}$ as 
$n\to\infty$.}
 \label{fig:H}
\end{figure}

Let us postpone for a moment the proof of \eqref{half-space} and conclude 
the proof of the theorem.
By \eqref{un>} and \eqref{half-space} we get
$$\forall T\geq 0,\ x\in \mc{H}_T,\quad u^*(-T,x)\geq\eta,$$
and, we recall, $u^*(0,0)=\eta$.
Roughly speaking, this means that the set $\{u^*\geq\eta\}$ expands in the direction $\nu(\hat x)$ with 
at most speed $k\,\hat x\.\nu(\hat x)$, which is smaller 
than the speed $c$ of the subsolution $v$ given by the hypothesis of the 
theorem.
In order to get a contradiction from this fact,
consider the function $u^*$ in the frame moving with speed $k(\hat 
x\.\nu(\hat x))$ in the direction $\nu(\hat x)$, i.e.,
$$\ol u(t,x):=u^*(t,x+\zeta t),\qquad\text{with}\quad
\zeta:=k(\hat x\.\nu(\hat x))\nu(\hat x).$$
The function $\ol u$ satisfies $\ol u(t,x)\geq\eta$ if $t\leq0$ and 
$x\.\nu(\hat x)<0$, together 
with $\ol u(0,0)=\eta$, and it is a solution of the equation
\Fi{past-moving}
\partial_t u-\dv(A^*(t,x+\zeta t)\nabla u)+[q^*(t,x+\zeta t)-\zeta]
\.\nabla u =f^*(t,x+\zeta t,\ol u),\quad t<0,\ x\in\R^N,
\Ff
The nonlinear term $f^*(t,x+\zeta t,s)$ clearly fulfils the first two 
conditions in \eqref{f-gen}. Moreover, as we have seen in the last part of the 
proof of Lemma \ref{lem:invasion}, $f^*$ inherits from $f$ the last condition 
in \eqref{f-gen}, and then the same is true for $f^*(t,x+\zeta t,s)$. 
Consequently, since $\ol u(t,x)\geq\eta>S$ 
for $t\leq0$ and $x\in H:=\{x\,:\,x\.\nu(\hat x)<0\}$, we can apply 
Lemma \ref{lem:invasion} and infer that $\ol u(t,x)\to1$ as $x\.\nu(\hat 
x)\to-\infty$ uniformly in $t\leq0$. This means that $\ol u$ satisfies 
\eqref{olu>} with $e=\nu(\hat x)$.
Let $v$ and $c>k \hat x\.\nu(\hat x)$ be as in the statement of the theorem.
The function $\ul u$ defined by 
$$\ul u(t,x):=v(t,x+\zeta t),$$ 
is a subsolution to \eqref{past-moving}. 
We want to apply Lemma \ref{lem:comparison} to $\ul u$, $\ol u$. To do this, we 
need to check that $\ul u$ satisfies \eqref{ulu<others} if the nonlinear term in 
\eqref{past-moving} fulfils \eqref{others}, or the stronger condition 
\eqref{ulu<mono} otherwise. Properties \eqref{ulu<others}, \eqref{ulu<mono} 
hold with  $e=\nu(\hat x)$ and $\gamma=c-k \hat x\.\nu(\hat x)>0$ if $v$ 
satisfies \eqref{v<others}, \eqref{v<mono} respectively. On the one 
hand, \eqref{v<others}, which is weaker than \eqref{v<mono}, always holds by 
hypothesis. On the other hand, if $f^*(t,x+\zeta t,s)$
does not fulfil \eqref{others} then neither does $f$, because \eqref{others} is 
preserved when passing to the limit of translations. Thus, in such case, $v$ 
satisfies \eqref{v<mono} by hypothesis.
We can thereby apply Lemma \ref{lem:comparison} and infer that $\ul 
u(0,0)\leq\ol u(0,0)$. 
This is a contradiction because $\ul u(0,0)=v(0,0)>\eta=\ol u(0,0)$.

To conclude the proof of the theorem, it remains to derive 
\eqref{half-space}. Take $x\in \mc{H}_T$. We compute
\[\begin{split}
\left|\frac{x+x_n}{\mc{R}(t_n)-kT}-y_n\right|
&=\frac{|x+x_n-(\mc{R}(t_n)-kT)
(\hat x_n-\rho\nu(\hat x_n))|}{\mc{R}(t_n)-kT}\\
&=\frac{|x+kT\hat 
x_n+(\mc{R}(t_n)-kT)\rho\nu(\hat x_n)|}{\mc{R}(t_n)-kT}\\
&= \left|\rho\nu(\hat x_n)+\frac{x+kT\hat x_n}
{\mc{R}(t_n)-kT}\right|.
\end{split}
\]
Calling $z_n:=(x+kT\hat x_n)/(\mc{R}(t_n)-kT)$, which tends to 0 
as $n\to\infty$, we rewrite the last term as
$$|\rho\nu(\hat x_n)+z_n|=\sqrt{\rho^2+2\rho\nu(\hat x_n)\.z_n+|z_n|^2}=
\sqrt{\rho^2+|z_n|(2\rho\nu(\hat x_n)\.z_n/|z_n|+|z_n|)}.$$
Since
$$
\limn(2\rho\nu(\hat x_n)\.z_n/|z_n|+|z_n|) = 2\rho\,
\frac{x\.\nu(\hat x)+kT\hat x\.\nu(\hat x)}{|x+kT\hat x|}<0,
$$
because $x\in \mc{H}_T$, we infer that, for sufficiently large $n$, $|\rho\nu(\hat 
x_n)+z_n|<\rho$ and thus 
$$\frac{x+x_n}{\mc{R}(t_n)-kT}\subset B_\rho(y_n)\subset\W.$$
Namely, $x+x_n\in (\mc{R}(t_n)-kT)\W$, and thus \eqref{half-space} is proved. 
\end{proof}

\begin{remark}
It follows from the proof of \thm{subset} that, for given $\hat 
x\in\partial\W$, the only limiting equations for which the 
existence of the subsolution $v$ is needed are the ones obtained by 
translations $(t_n,x_n)_{n\in\N}$ such that
$$\lim_{n\to\infty} 
t_n=+\infty,\qquad\lim_{n\to\infty}\frac{x_n}{|x_n|}=\frac{\hat x}{|\hat x|}.$$
\end{remark}


%


\subsection{Application of the general result}

We now prove \thm{intermediate}. We cannot apply
\thm{subset} directly to the set $\W$ defined by \eqref{Wsub_general}
because it may not fulfil the
uniform interior ball condition. The idea is to consider an interior smooth
approximation $\widetilde{\W}$ of $\W$, but to this end we need at least the 
function $w$ defining $\partial\W$ to be continuous. This is a
general consequence of the definition of $w$. 

\begin{proposition}\label{pro:wC}
 Let $c:\Sph\to\R$ satisfy $\inf c>0$. Then the function $w:\Sph\to\R$ defined
by 
$$w(\xi):=\inf_{e\.\xi>0}\frac{c(e)}{e\.\xi},$$
is positive and continuous.
\end{proposition}

\begin{proof}
There holds the lower bound $w\geq\inf c>0$. Let us show that
$w$ is bounded from above. Consider the family
$\mc{B}:=\{\pm e_1,\pp\pm e_N\}$, where $\{e_1\pp e_N\}$ is the canonical 
basis of $\R^N$. Then, calling $\ol c:=\max_{\mc{B}}c$, we find
$$\forall\xi\in S^{N-1},\quad
w(\xi)\leq\min_{\su{e\in\mc{B}}{e\.\xi>0}}\frac{c(e)}{e\.\xi}
\leq\ol c\left(\max_{\su{e\in\mc{B}}{e\.\xi>0}}\,e\.\xi\right)^{-1}\leq \ol 
c\sqrt{N}.$$
Now, fix $\xi\in S^{N-1}$. For $\e\in(0,1)$, let $e_\e\in S^{N-1}$ be such that 
$$e_\e\.\xi>0,\qquad w(\xi)>\frac{c(e_\e)}{e_\e\.\xi}-\e.$$
Hence, $c(e_\e)/e_\e\.\xi<\ol c\sqrt{N}+1$, from which we deduce
$$c(e_\e)<\ol c\sqrt{N}+1,\qquad e_\e\.\xi>h:=
\frac{\inf c}{\ol c\sqrt{N}+1}.$$
For $\xi'\in S^{N-1}$ such that $|\xi'-\xi|<h/2$, it holds that 
$e_\e\.\xi'>h/2$, whence
$$w(\xi')-w(\xi)\leq \frac{c(e_\e)}{e_\e\.\xi'}-w(\xi)<
\frac{c(e_\e)}{e_\e\.\xi'}-\frac{c(e_\e)}{e_\e\.\xi}+\e
\leq2\,\frac{\ol c\sqrt{N}+1}{h^2}|\xi-\xi'|+\e.$$
The latter term is smaller than $2\e$ for $|\xi'-\xi|$ small enough, 
independently of $\xi,\xi'$. This shows that $w$ is (uniformly) continuous.
\end{proof}

\begin{proof}[Proof of \thm{intermediate}]
Let $w$ and $\W$ be as in \eqref{Wsub_general}.
Owing to \eqref{infc>0}, we can apply Proposition \ref{pro:wC} and 
deduce that $w$ is positive and continuous. It follows in particular that 
$\min w>0$ and that the set $\W$ coincides with the closure of its interior.
Moreover, for any $h\in(0,\min w)$, we can 
consider a smooth approximation $\t w$ of the function $w-h/2$ satisfying 
$w-h<\t w<w$.
If we show that, for any $h\in(0,\min w)$, the set 
$$
\widetilde{\mc{W}}:=\{r\xi\,:\,\xi\in S^{N-1},\ \ 0\leq r\leq \t w(\xi)\},
$$
is an \sub, the same is true for $\W$, because
if $K\Subset\inter(\W)$ then $K\Subset\inter(\widetilde{\mc{W}})$ for 
$h$ small enough.
This is achieved by showing that $\widetilde{\W}$ satisfies the hypotheses of 
\thm{subset}.


Consider $\eta,k<1$, $\hat x\in\partial\widetilde{\W}$, the (unique) exterior 
unit normal $\nu(\hat x)$ to $\widetilde{\W}$ and a limiting
equation \eqref{le}. We know that $\hat x\neq0$ and thus we can write $\hat
x=\t w(\xi)\xi$, with $\xi:=\hat x/|\hat x|\in S^{N-1}$. 
%
By hypothesis, there is a transition front $v$ in the direction $\nu(\hat x)$
for \eqref{le} on $\R_-\times\R^N$, which connects $0$ and $1$ if $f$ satisfies 
\eqref{others}, or some $-\e<0$ and~$1$ otherwise, has speed larger 
than $c:=k\ul c(\nu(\hat x))$ and satisfies $v(0,0)>\eta$.
Let $X$ be the function for which $v$ satisfies the limits in \eqref{gtf} with 
$S_1=-\e$ or $0$ and $S_2=1$. It follows from the uniformity of these limits 
and the strong maximum principle that $v<1$.
Moreover, since $v$ has speed larger than $c$, there holds that $X(t)<ct$ for 
$t$ less than some $T<0$. On the other hand, we know from \cite{HR1} that $X$ 
is locally bounded and thus there exists $K>0$ such that $X(t)<ct+K$ for all 
$t\leq0$. As a consequence, by \eqref{gtf}, $v$ satisfies \eqref{v<others} if 
$f$ fulfils \eqref{others}, or \eqref{v<mono} otherwise.
%
%
%
Finally, we deduce from the smoothness of $\t w$ that $\hat 
x\.\nu(\hat 
x)>0$, 
i.e.~$\xi\.\nu(\hat x)>0$. We can then compute
$$c=k\ul c(\nu(\hat x))\geq k\nu(\hat x)\.\xi\,\inf_{\su{e\in 
S^{N-1}}{e\.\xi>0}}\frac{\ul c(e)}{e\.\xi}=
k\nu(\hat x)\.\xi\, w(\xi)>k\t w(\xi)\,\xi\.\nu(\hat x)=k\hat 
x\.\nu(\hat x).$$
We have shown that $v$ satisfies all the requirements in \thm{subset}, whence 
$\widetilde{\W}$ is an \sub.
%
%
\end{proof}


\subsection{The periodic case}\label{sec:periodic}

In this subsection, we prove that the set $\W$ defined by 
\eqref{W} is an \sub\ for \eqref{future}. This is achieved by showing that the 
(minimal) speed $c^*$ for 
pulsating travelling fronts satisfies the hypotheses for $\ul c$ in 
\thm{intermediate}.

We recall the known results about \PTF s: they are increasing in time and their 
critical speed $c^*(e)$ is positive. It is also readily seen 
that~\eqref{ptf} yields the transition front condition \eqref{gtf} with 
$S_1\!=0$, $S_2\!=1$ and $X(t)=cte$.

The first hypothesis to check in \thm{intermediate} is $\inf c^*>0$. We derive 
it from the following result, which is of independent interest. 

\begin{proposition}\label{pro:c*LSC}
Under the assumptions of \thm{Freidlin}, the function $c^*:S^{N-1}\to\R$ is 
lower semicontinuous.
\end{proposition}

\begin{proof}
We need to show that, 
given a sequence $\seq{e}$ in $\Sph$ such that 
$$e_n\to e\in\Sph\quad\text{and}\quad c^*(e_n)\to c\in[0,+\infty)\quad\text{as 
}n\to\infty,$$ 
there holds that $c^*(e)\leq c$. Let $v_n$ be the
pulsating travelling front in the direction~$e_n$ connecting $0$ and $1$ with 
speed $c^*(e_n)$. Take $M\in(\theta,1)$ if $f$ satisfies either \eqref{igni} or 
\eqref{bi}, or set $M:=1/2$ in the case \eqref{mono}. Since $v_n(t,x)\to0$ or 
$1$ as $t\to-\infty$ or $+\infty$ locally uniformly in $x\in\R^N$, by a 
temporal translation we reduce to the case where
\Fi{vn>M}
\min_{x\in[0,1]^N}v_n(0,x)=M.
\Ff
The $v_n$ converge (up to subsequences) locally uniformly to a solution 
$0\leq v\leq 1$
which is nondecreasing in $t$ and satisfies the normalization condition 
\eqref{vn>M}.
Actually, $0<v<1$ by the parabolic strong \MP.

{\em Case $c>0$.}\\
Because the $v_n$ satisfy the first condition in \eqref{ptf} with $e=e_n$ and 
$c=c_n$, passing to the limit as $n\to\infty$ we deduce that $v$ satisfies the 
first condition in \eqref{ptf}. Then, letting $t\to\pm\infty$ in such condition 
we infer that the functions $v^\pm$ defined by $v^\pm(x):=v(\pm\infty,x)$ are 
$1$-periodic. It 
follows in particular that 
$$\exists x^\pm\in\R^N,\quad v^\pm(x^\pm)=\min_{\R^N}v^\pm=:m^\pm,\qquad
0\leq m^-\leq M\leq m^+\leq1.$$ 
We further know from parabolic estimates that the 
convergences of $v$ to $v^\pm$ as $t\to\pm\infty$ hold locally uniformly 
in $\R^N$, and that the $v^\pm$ 
are stationary solutions of~\eqref{future}. 
Since $f\geq0$ on $\R^N\times[M,1]$, we have that $f(x,v^+)\geq0$. The \SMP\ 
then yields $v^+\equiv m^+$ and thus $f(x,m^+)=0$ for all $x\in\R^N$.
We then deduce from the choice of $M$ that $m^+=1$, that is, $v^+\equiv1$. 
For $x\in\R^N$, let $z(x)\in\Z^N$ be such that $x-z(x)\in[0,1)^N$. 
By the first property in \eqref{ptf}, we can write
\Fi{vmoving}
\forall (t,x)\in\R\times\R^N,\quad
v(t,x)=v\Big(t-\frac{z(x)\.e}c,x-z(x)\Big), \quad\text{with }\ x-z(x)\in[0,1)^N.
\Ff
Whence, since $z(x)\.e\to-\infty$ as $x\.e\to-\infty$ and 
$v(+\infty,x)=1$ locally uniformly in $x$, we find that
\Fi{vto1}v(t,x)\to1\quad\text{as 
}x\.e\to-\infty\quad\text{locally uniformly in }
t\in\R.\Ff
Similarly, if $f$ satisfies the monostability hypothesis \eqref{mono}, 
we derive $m^-=0$, whence $v^-\equiv0$ by the 
strong maximum principle, and \eqref{vmoving} eventually yields 
that both limits in the last condition in \eqref{ptf} hold in such case. That 
is, $v$ is a pulsating travelling front in the direction $e$ connecting $0$ and 
$1$ 
with speed $c$, and therefore $c^*(e)\leq c$ by definition.
If instead $f$ is of either combustion or bistable type, i.e.~\eqref{igni} or 
\eqref{bi} hold, we cannot deduce $m^-=0$ and infer that $v$ connects $0$ and 
$1$.
In these cases we resort to Lemma \ref{lem:comparison}.
Set  $\ol u(t,x):=v(t,x+cte)$. For $t\in\R$, letting $z(t)\in\Z^N$ be such that 
$cte-z(t)\in[0,1)^N$ and using the fact that $v$ verifies the first 
condition in~\eqref{ptf}, we~get, for $x\in\R^N$,
$$\ol u(t,x)=v\Big(t-\frac{z(t)\.e}c,x+cte-z(t)\Big)
=v\Big(\frac{cte-z(t)}c\.e,x+cte-z(t)\Big).$$
Hence, by \eqref{vto1}, $\ol u$ satisfies $\ol u(t,x)\to1$ as $x\.e\to-\infty$ 
uniformly in $t\in\R$, 
and then in particular \eqref{olu>}.
Next, consider the pulsating travelling front~$\t v$ in 
the direction $e$ connecting $0$ and $1$ 
(with speed $c^*(e)$), translated in time in such a way that 
$\t v(0,0)>v(0,0)$, and set $\ul u(t,x):=\t v(t,x+cte)$.
If we had $c<c^*(e)$, since
$$\ul u(t,x)=\t v(t,x+c^*(e)te-(c^*(e)-c)te)$$
and $\t v$ satisfies \eqref{gtf} with $X(t)=c^*(e)t$ and $S_1=0$ (and 
$S_2=1$), condition \eqref{ulu<others} would be fulfilled with 
$\gamma=c^*(e)-c>0$.
We could then apply Lemma \ref{lem:comparison} to $\ul u$ 
and $\ol u$, which satisfy \eqref{future} for $t<0$ with $q$ replaced by $q+ce$,
and 
deduce $\t v\leq v$ in $\R_-\times\R^N$, in contradiction with $\t 
v(0,0)>v(0,0)$.
Hence, $c^*(e)\leq c$ in cases \eqref{igni}, \eqref{bi} too.

{\em Case $c=0$.}\\
The $v_n$ satisfy \eqref{vmoving} with $c=c_n$, $e=e_n$. For $x\.e<-\sqrt{N}$ 
it holds that $z(x)\.e<0$, whence, for any $t\in\R$, $t-z(x)\.e_n/c_n>0$ for 
$n$ large enough because $c_n\searrow0$. It then follows from the fact that the 
$v_n$ are increasing in time and from \eqref{vn>M} that, for $t\in\R$ and 
$x\.e<-\sqrt{N}$,
$$v(t,x)=\limn 
v_n\left(t-\frac{z(x)\.e}{c_n},x-z(x)\right)
\geq \limn v_n(0,x-z(x))\geq M.$$
%
Thus, by Lemma \ref{lem:invasion}, $v(t,x)\to1$ as 
$x\.e\to-\infty$ uniformly in $t\leq0$, and then $\ol u=v$ fulfils 
\eqref{olu>}.
If $f$ satisfies either \eqref{igni} or \eqref{bi}, we get a 
contradiction as before applying Lemma \ref{lem:comparison} with $\ol u=v$ and 
$\ul u$ 
equal to the pulsating travelling front~$\t v$ in the direction $e$ connecting 
$0$ and $1$.
Suppose that $f$ satisfies \eqref{mono}. Setting $f(x,s)=0$ for $s<0$, we have 
that $f$ is of combustion type if considered on, say, $\R^N\times[-1,1]$. 
Namely, 
it satisfies hypothesis \eqref{igni} up to an affine transformation of the 
second variable. There exists then a pulsating travelling front in the 
direction $e$ connecting $-1$ and $1$ with a speed $c'>0$.
Let $\ul u$ be this front, normalized in such a way that $\ul u(0,0)>v(0,0)$.
It is an entire solution to \eqref{future} satisfying \eqref{ulu<mono}. We
therefore 
get a contradiction applying once again Lemma \ref{lem:comparison} with such 
$\ul u$ and $\ol u=v$.
\end{proof}

\begin{proposition}\label{pro:sub}
Under the assumptions of \thm{Freidlin}, the function $w$ defined
in~\eqref{W} is positive and continuous and $\W$ is  an \sub\
for~\eqref{future}.
\end{proposition}

\begin{proof}
The positivity and continuity of $w$ follow from Propositions
\ref{pro:wC}, \ref{pro:c*LSC} and the fact that $c^*$ is positive.
In order to apply \thm{intermediate} with $\ul c=c^*$, it remains to check the 
hypothesis concerning the existence of the pulsating travelling front $v$. To 
this end, 
fix $e\in S^{N-1}$, $c<c^*(e)$, $\eta<1$ and consider a limiting equation 
\eqref{le} associated with \eqref{future}. By periodicity, the coefficients 
of such equation are simply translations of $A,q,f$ by the same 
$\zeta\in[0,1)^N$. We can assume without loss of generality that
$\zeta=0$.

In the case where $f$ is of combustion type \eqref{igni} or bistable type \eqref{bi}, we take $v$ equal to the pulsating 
travelling front connecting $0$ and $1$ in the direction $e$, normalized in 
such a way that $v(0,0)>\eta$.

The monostable case \eqref{mono} is more involved.
Let $v^*$ be a pulsating travelling front connecting $0$ and $1$ in the 
direction $e$ with (the minimal) speed $c^*(e)$.
For $\e>0$, the nonlinearity $f:\R^N\times[-\e,1]\to\R$ is of combustion 
type and therefore there exists a unique $c_\e>0$ for which \eqref{future} 
admits a pulsating travelling front $v_\e$ in the direction $e$ connecting $-\e$ 
and $1$. We will show that 
\Fi{approxigni}
c_\e\nearrow c^*(e)\quad\text{ as }\ \e\searrow0.
\Ff
A similar property is proved in \cite{pulsating} using some estimates on the 
first derivatives of the fronts. Let us present a direct approach based on the 
comparison result of Lemma~\ref{lem:comparison}.
Recalling that $v^*$ and $v_\e$ satisfy \eqref{gtf} with $X(t)=c^*(e)t\,e$, 
$S_1=0$, $S_2=1$ and with $X(t)=c_\e t\,e$, $S_1=-\e$, $S_2=1$ respectively, we 
see that, if we had $c_\e>c^*(e)$ for some $\e>0$, Lemma~\ref{lem:comparison} 
would apply with $q$ replaced by $q+ce$ in equation \eqref{past-gen} and
$$\ol u(t,x)=v^*(t,x+c^*(e) te),\qquad \ul 
u(t,x)=v_\e(t,x+c^*(e) te),$$ 
yielding $v^*\geq v_\e$ in 
$\R_-\times\R^N$. This is impossible because, up to a suitable temporal 
translation, we can always reduce to the case where $v^*<v_\e$ at, say, 
$(0,0)$. Hence $c_\e\leq c^*(e)$.
It is clear that the conclusion of Lemma~\ref{lem:comparison} holds true
if the conditions $\ol u>0$ in \eqref{olu>}
and $\forall s>0$ in  \eqref{ulu<others}  are replaced by $\ol u>-\e$ and 
$\forall s>-\e$. We can therefore argue as before and infer that if 
$c_\e<c_{\e'}$ for some $0<\e<\e'$, then $v_\e\geq v_{\e'}$ in $\R_-\times\R^N$, 
and this is a contradiction up to a temporal translation of $v_\e$ or $v_{\e'}$.
As a consequence, $c_\e$ decreases to some value $c_0\in(0,c^*(e)]$ 
as $\e\searrow0$. Let us normalize the $v_\e$ by $v_\e(0,0)=1/2$.
As $\e\to0$, the $v_\e$ converge (up to subsequences) locally uniformly to an 
entire solution $v_0$ of \eqref{future} satisfying
$$v_0(0,0)=1/2,\qquad 0\leq v_0\leq1,\qquad \partial_t v\geq0.$$
Moreover, $v_0$ satisfies the first condition in \eqref{ptf} with $c=c_0$.
Then, the second condition follows exactly as in the case $c>0$ of the proof of 
the Proposition \ref{pro:c*LSC}. This means that $v_0$ is a pulsating 
travelling front
connecting $0$ and $1$ in the direction $e$, which implies that $c_*(e)\leq 
c_0$ by definition and concludes the proof of \eqref{approxigni}.
Finally, by \eqref{approxigni}, we can choose $\e>0$ small enough in such a way 
that
$c_\e\in(c,c_*(e)]$, and then the associated front $v_\e$, translated in $t$ in order to have $v_\e(0,0)>\eta$, satisfies the desired properties for $v$.
This concludes the proof of the proposition.
\end{proof}


\section{Asymptotic superset of spreading}\label{sec:super}



\begin{proof}[Proof of \thm{super}]
Let $w$ and $\W$ be as in \eqref{W+gen}. Because of \eqref{infc>>0}, 
Proposition~\ref{pro:wC} implies that $w$ is positive and continuous. It 
follows in 
particular that $\W$ coincides with the closure of its interior. It remains to 
verify that \eqref{W+} holds for any $u$ as in Definition \ref{def:W}. It is 
well known (see, e.g., \cite{Friedman})
that $u(1,x)$ decays as a Gaussian as $|x|\to\infty$, 
because the initial datum $u_0$ has compact support. 
Call
$$\eta:=\max_{x\in\R^N}u(1,x),$$
which is strictly less than $1$ by the parabolic strong maximum principle.
Take $e\in\Sph$. Let $v$ be the front given by the hypothesis of the 
theorem, associated with $e$,~$\eta$ and some~$R\geq0$ to be chosen.
Because of condition \eqref{v>}, the function~$v$ decays at most exponentially 
in the direction $e$. Namely, applying \cite[Lemma~3.1]{RR} to the function 
$\phi(t,x):=v(t,x+[R-1+\ol c(e)t] e)$ we infer the existence of a constant 
$\lambda>0$, only depending on $A,q,e,\ol c(e)$, such that
$$\forall x\.e>R+\ol c(e),\qquad
v(1,x)\geq \eta e^{-\lambda(x\.e-R+1-\ol c(e))}\geq \eta e^{-\lambda(x\.e+1)}.$$ 
Since, on the other hand, $v(1,x)\geq \eta$ for $x\.e\leq R+\ol c(e)$,
it follows from the Gaussian decay of $u(1,\.)$ that, choosing $R$ large 
enough, the  front $v$ satisfies 
$v(1,\.)\geq u(1,\.)$ in the whole $\R^N$.
As a consequence of the comparison principle we thus infer that $u\leq v$ for 
all $t\geq1$, whence, since $v$ satisfies~\eqref{gtf} with $S_1=0$, and 
$\limsup_{t\to+\infty} X(t)/t\leq \ol c(e)$, we get
\Fi{uto0}
\forall c>\ol c(e),\quad \sup_{x\.e>ct}u(t,x)\to0\quad\text{as 
}t\to+\infty.
\Ff 
Using this property in different directions $e$ one easily derives 
\eqref{c>w} with $w$ as in \eqref{W+gen}. But the uniform version of 
\eqref{c>w}, property \eqref{W+}, requires some additional work and in 
particular the continuity of $w$. We proceed as follows.

Fix $\e>\!0$ and $\xi\in\Sph$. By the definition of $w$ in \eqref{W+gen}, 
there 
is $e\in\Sph$ such~that $e\.\xi>0$ and $\ol c(e)/e\.\xi<w(\xi)+\e/3$.
For $\xi'\in\Sph$ close enough to $\xi$, there holds
$$(w(\xi)+2\e/3)\xi'\.e>(w(\xi)+\e/3)\xi\.e>\ol c(e).$$
Hence, by the continuity of $w$,
$(w(\xi')+\e)\xi'\.e>\ol c(e)$ provided $\xi'$ is in a small neighbourhood 
$U_\xi\subset\Sph$ of $\xi$. We can therefore make use of \eqref{uto0} and 
derive
$$\sup_{\su{\xi'\in U_\xi}{r\geq w(\xi')+\e}}
u(t,rt\xi')\to0\quad\text{as }t\to+\infty.$$
By compactness, there is a finite covering of $\Sph$ by sets of the type 
$U_\xi$, $\xi\in\Sph$, whence the above limit actually holds taking the $\sup$ 
among all $\xi'\in\Sph$. This concludes the proof of \eqref{W+}, because if 
$C$ is a closed set such that $C\cap\W=\emptyset$, then
$C\subset\{r\xi'\,:\,\xi'\in\Sph,\ r\geq w(\xi')+\e\}$ with $\e=\dist(C,\W)$.
\end{proof}

\begin{proof}[Proof of \thm{Freidlin}]
Let $\W$ be defined by \eqref{W}. By Proposition \ref{pro:sub} we know that 
$\W$ is an \sub\ for \eqref{future}. It remains to show that it is an \super\ 
too.
This is achieved using \thm{super}, showing that the minimal speed 
for \PTF s $c^*$ fulfils the hypotheses for $\ol c$ there. 
We already know that $\min c^*>0$ because $c^*$ is positive and it is lower 
semicontinuous by Proposition \ref{pro:c*LSC}. Fix 
$e\in\Sph$ and let $v$ be the \PTF\ in the direction $e$ connecting $0$ and $1$ 
with speed $c^*(e)$. We know from \eqref{ptf} that $v$ satisfies the transition 
front condition 
\eqref{gtf} with $S_1=0$, $S_2=1$ and $X(t)=c^*(e)t\,e$. Hence, for 
any $\eta<1$, there exists $L\in\R$ such that
$$\forall t\in\R,\ x\.e<L,\quad v(t,x+c^*(e)te)>\eta.$$
For given $R>1$, let $z\in\Z^N$ be such that $z\.e<L-R$. Hence, the 
translation $v^z$ of $v$ defined by $v^z(t,x):=v(t,x+z)$, which is still a 
\PTF\ for \eqref{future} with speed $c^*(e)$, satisfies
$$\forall t\in\R,\ x\.e-c^*(e)t\leq R,\quad
v^z(t,x)=v(t,x+z)>\eta,$$
because $(x+z-c^*(e)te)\.e\leq R+z\.e<L$. It follows that $v^z$ fulfils 
\eqref{v>}. We can therefore apply \thm{super} and conclude the proof.
\end{proof}


\section{Almost periodic, time-dependent equations}\label{sec:Shen}

In this section we deduce Corollary \ref{cor:Shen} from Theorems \ref{thm:intermediate}, \ref{thm:super}.
Here are the assumptions under which Shen derives the existence of fronts in
\cite{Shen-comb-ap} and \cite{Shen-bi2} respectively.
\begin{description}
\item[{ \normalfont\em Combustion\,: }]
$\exists \theta\in(0,1)$, $f(t,s)=0$ for $s\leq\theta$ and $s=1$,  $f(t,s)>0$ for $s\in(\theta,1)$.\\
	$f$ is of class $C^1$ with respect to $s\in[\theta,1]$ and there satisfies: $\inf_t\partial_s f(t,\theta)>0$,
	$\sup_t\partial_s f(t,1)<0$.
\item[{ \normalfont\em Bistable\,: }]
the equation $\vt'(t)=f(t,\vt(t))$ in $\R$ admits an a.p.~solution $0<\theta(t)<1$, and any other solution
satisfies $\vt(+\infty)=0$ if $\vt(0)<\theta(0)$ and $\vt(+\infty)=1$ if $\vt(0)>\theta(0)$.\\
$f\in C^2$ and its derivatives up to order 2 are a.p.~in $t$ uniformly in $s$.\\
$\sup_t\partial_s f(t,0)<0,\quad\sup_t\partial_s f(t,1)<0,\quad
\inf_t\partial_s f(t,\theta(t))>0.$
\end{description}

\begin{proof}[Proof of Corollary \ref{cor:Shen}]
For $e\in S^{N-1}$, let $v=v(t,x\.e)$ be the planar front provided by \cite{Shen-comb-ap,Shen-bi2} and~$X$ be the 
associated function in Definition \ref{def:gtf}. The functions
$X'(t)$ and $v(t,x\.e+X(t))$ are a.p.~in $t$ uniformly in~$x$, and $X'$ has uniform average $c^*$ in the sense of 
\eqref{avspeed}.
We want to show that the hypotheses of Theorems \ref{thm:intermediate}, \ref{thm:super} are fulfilled with $\ul c$ and $\ol c$ constantly equal to t $c^*$.
The front $v(t,x\.e)$ is a transition front in the direction $e$ with
future (and past) speed equal to $c^*$. Moreover, because of the space-invariance of the equation \eqref{t-dep}, we can translate $v$ in such a way that it fulfils~\eqref{v>} for any given $\eta<1$ and $R\in\R$.
Then, if $c^*>0$, \thm{super} implies that $\ol B_{c^*}$ is an \super.
Observe that if $c^*\leq0$ (which is possible in the bistable case), the same comparison argument as in the proof of \thm{super} implies that no solution with compactly supported datum can converge to $1$ as $t\to+\infty$, and thus 
the definition of asymptotic set of spreading is vacuously satisfied by any set.

It remains to show that $\ul c\equiv c^*$ satisfies the hypotheses of 
\thm{intermediate}. This is a consequence of the fact that, by the almost periodicity,
the limit of translations of a front preserves the average speed.
Consider indeed a limiting equation associated with~\eqref{t-dep}. By the almost periodicity of $f$, this equation is of the form $\partial_t u-\Delta u = f^*(t,u)$,
with $f^*$ obtained as the uniform limit of time-translations of $f$ by some diverging sequence $\seq{t}$.
By a priori estimates, the translations of the front $v(t+t_n,x\.e+X(t_n))$ converge (up to subsequences) locally uniformly to a solution $v^*(t,x\.e)$ of the limiting equation. On the other hand, the almost periodicity implies the existence
of $w$ and $c$ 
such that, as $n\to\infty$ (up to subsequences), there holds 
$$v(t+t_n,x\.e+X(t+t_n))\to w(t,x\.e),\qquad X'(t+t_n)\to c(t),$$
{\em uniformly} in $t\in\R$ and $x\in\R^N$, with $w(t,-\infty)=1$ and $w(t,+\infty)=0$ uniformly in $t\in\R$. 
Next, calling $Y(t):=\int_0^t c(s)ds$ we find that
\[\begin{split}
v^*(t,x\.e+Y(t))&=\limn v(t+t_n,x\.e+Y(t)+X(t_n))\\
&=\limn v(t+t_n,x\.e+\int_0^t\big(c(s)-X'(s+t_n)\big)ds+X(t+t_n))\\
&=w(t,x\.e),\end{split}\]
from which we deduce that $v^*$ is a transition front in the sense of Definition \ref{def:gtf} with $X=Y$. Finally, 
\eqref{avspeed} yields
$$\lim_{t\to\pm\infty}\frac{Y(t)}t=\lim_{t\to\pm\infty}\frac1t\int_0^t c(s)ds=
\lim_{t\to\pm\infty}\limn\frac1t\int_{t_n}^{t_n+t} X'(s)ds=c^*,$$
whence $v^*$ has past speed $c^*$.
It follows that, up to a suitable spatial translation, $v^*$ satisfies 
the hypotheses of \thm{intermediate} with $\ul c(e)=c^*$. 
\end{proof}


\section{Proof of the Metatheorem}\label{sec:meta}

In this section we prove the Metatheorem stated in the introduction.
The hypotheses required on the operator will be pointed out during the proof, marked by~``$\bullet$''.

Consider an equation 
\Fi{meta}
\mc{P}u=0,\quad t>0,\ x\in\R^N.
\Ff

\begin{itemize}
	\item For any $e\in S^{N-1}$, solutions with front-like initial data \eqref{f-l}
	admit the same \ass\ $c^*(e)$ in the direction $e$, and $\inf_{e}c^*(e)>0$.
\end{itemize}
Consider the Wulff shape of the function $w$ given by the Freidlin-G\"artner formula:
$$\W:=\{r\xi\,:\,\xi\in S^{N-1},\ \ 0\leq r\leq w(\xi)\},\qquad\
\text{with}\quad w(\xi):=\inf_{e\.\xi>0}\frac{c^*(e)}{e\.\xi}.$$
It follows from Proposition \ref{pro:wC} that $w$ is positive and continuous, whence
$\W$ coincides with the closure of its interior.
Let $u$ be a solution of~\eqref{meta} with a compactly supported initial datum $0\leq u_0\leq1$ such that $u(t,x)\to1$ as $t\to+\infty$
locally uniformly in $x\in\R^N$. 
\begin{itemize}
	\item The operator $\mc{P}$ satisfies the comparison principle.
\end{itemize}
The fact that $\W$ is an asymptotic superset of spreading is shown 
by comparing $u$ with solutions with front-like initial data\,\footnote{\ 
	The argument here is simpler because we are assuming $u_0<1$.}
and using the continuity of $w$, exactly as in Section \ref{sec:super}. 

Assume by contradiction that $\W$ is not an \sub. Recall that the definition \eqref{W-}
of \sub\ involves compact subsets of~$\W$. Thus, up to slightly shrinking $\W$ as in the proof of \thm{intermediate}, it is not restrictive to assume that $\W$ is smooth. 
\begin{itemize}
	\item $\mc{P}$ is space-time periodic: $(\mc{P}u)\circ\mc{T}_{Z}=\mc{P}(u\circ\mc{T}_{Z})$ for all 
	$Z\in\Z^{N+1}$, where $\mc{T}_Zu(t,x):=u((t,x)+Z)$.
	\item Sequences of solutions of \eqref{meta} satisfy a-priori estimates allowing to pass to the limit in the equation (up to subsequences).
\end{itemize}
Using a-priori estimates and proceeding exactly as in the proof of \thm{subset},
for any given $\eta<1$, we find a sequence of translations of $u$ converging to a function $u^*$ which is an {\em entire} solution of some limit equation $\mc{P}^*u^*=0$ and satisfies, in addition,
\Fi{u*f-l}
u^*(0,0)=\eta,\qquad
\forall T\geq 0,\ x\in \mc{H}_T, \quad u^*(-T,x)\geq\eta,
\Ff
where $\mc{H}_T$ is the half-space
$$\mc{H}_T:=\{x\in\R^N\,:\,x\.\nu(\hat x)<-k(\hat x\.\nu(\hat x))T\},$$
for some $k\in(0,1)$, $\hat x\in\partial\W$. 
By the periodicity of $\mc{P}$, the limit operator $\mc{P}^*$ is simply the translation of $\mc{P}$
by some $(\bar t,\bar x)\in[0,1)^{N+1}$, namely, $u^*$ is a solution of the equation
$\mc{P}(u^*\circ\mc{T}_{(-\bar t,-\bar x)})=0$.
Call $e:=\nu(\hat x)$ and, writing $\hat x=w(\xi)\xi$ with $\xi=\hat x/|\hat x|\in S^{N-1}$, 
$$c:=k\hat x\.e=k w(\xi)\xi\.e\leq kc^*(e)<c^*(e).$$ 
Let $\seq{z}$ in $\Z^N$ be such that 
$z_n-c(\bar t+n)e\in[0,1)^N$.
We then consider  the sequence of functions $\seq{u^*}$ defined by 
$u^*_n(t,x):=u^*(t-\bar t-n,x-\bar x-z_n)$. 
They are entire solutions of \eqref{meta}.
In order to rewrite the second condition in \eqref{u*f-l} in terms of $u^*_n$ we observe that, for $x\.e<\bar x\.e-\sqrt N$, there holds 
$$(x-\bar x-z_n)\.e<-z_n\.e-\sqrt N<-c(\bar t+n).$$
Hence, \eqref{u*f-l} rewrites
\Fi{u*nf-l}
u^*_n(\bar t+n,\bar x+z_n)=\eta,\qquad
\forall x\.e<\bar x\.e-\sqrt N, \quad u^*_n(0,x)\geq\eta.
\Ff
We claim that, taking $\eta\in(\max u_0,1)$, the $\seq{u^*}$ are {\em uniformly} front-like in the direction $e$, in the following sense:
$$\liminf_{x\.e\to-\infty}u^*_n(0,x)\geq 1\quad\text{ uniformly in }n\in\N.$$
Let us prove the claim.
Since $u$ invades, for any $\e>0$, there exists $m\in\N$ such that $u(m,0)>1-\e$.
 Let $R>0$ be such that $u_0(x)=0$ for $x\.e\geq R$. Then, 
 for  $\zeta\.e\geq R-\bar x\.e+\sqrt N$, since on one hand $u_0(x+\zeta)=0$ if
 $x\.e\geq\bar x\.e-\sqrt N$, and on the other $u^*_n(0,x)\geq\eta>\max u_0$ 
 if $x\.e<\bar x\.e-\sqrt N$ by~\eqref{u*nf-l}, we derive
 $$\forall n\in\N,\ x\in\R^N,\quad u_0(x+\zeta)\leq u^*_{n+m}(0,x).$$
Hence, recalling that the $\seq{u^*}$ are solutions of
\eqref{meta}, the comparison principle yields
$$\forall n\in\N,\ \zeta\.e\geq R-\bar x\.e+\sqrt N,\quad u^*_{n+m}(m,-\zeta)
\geq u(m,0)>1-\e.$$
Since $u^*_{n+m}(m,-\zeta)=u^*_n(0,-\zeta-z_{n+m}+z_n)$ and
$|z_n-z_{n+m}|<\sqrt N$, we deduce
$$\forall n\in\N,\ x\.e\leq-R+\bar x\.e-2\sqrt N,\quad u^*_n(0,x)>1-\e.$$
The claim is thereby proven. 
Then, using the claim, we can find a continuous function $\chi:\R\to[0,1]$ such that $\chi(-\infty)=1$, 
$\chi=0$ in $\R_+$ and
$$\forall n\in\N,\ x\in\R^N,\quad
u^*_n(0,x)\geq\chi(x\.e).$$
Owing again to the comparison principle, the solution $v$ of \eqref{meta} with initial datum $\chi(x\.e)$ satisfies
$v(t,x)\leq u^*_n(t,x)$ for all $n\in\N$, $t>0$, $x\in\R^N$.
We now use the hypothesis about solutions with front-like initial data.
Namely,  since  $\chi(x\.e)$ is a front-like initial datum in the sense of \eqref{f-l},
we know that $v$ has \ass\ $c^*(e)$ in the direction $e$. Thus, because $c\in[0,c^*(e))$, we have that $v(t,x+cte)\to1$ as $t\to+\infty$, 
locally uniformly in $x\in\R^N$. Consequently,
$$u^*_n(\bar t+n,\bar x+z_n)\geq 
v(\bar t+n,\bar x+z_n)\to 1\quad\text{as }n\to\infty,$$
because $\bar x$ and $z_n-c(\bar t+n)e$ belong to $[0,1)^N$.
This contradicts the first condition in~\eqref{u*nf-l}. The proof of the Metatheorem
is thereby complete.


\section{Multilevel speeds of propagation}\label{sec:terrace}

This section is dedicated to the proof of \thm{terrace}. Observe preliminarily 
that the hypothesis \eqref{DGM} yields
\Fi{f>0}
\exists S\in(0,1),\quad f>0\quad\text{in }(S,1),
\Ff
because otherwise there would exist constant supersolutions of \eqref{homo}
arbitrarily close to $1$, preventing any solution smaller than $1$ from invading.

The quantities in \thm{terrace} come from the one-dimensional result of \cite{Terrace}.
This asserts that, under the hypotheses \eqref{hyp-homo}-\eqref{DGM},
the solution $v$ of \eqref{homo} in dimension $N=1$ with initial datum 
$v_0=\1_{(-\infty,0]}$ converges to the minimal {\em propagating terrace} 
connecting $1$ to $0$.
As a consequence, there exist $\t M\in\N$, $0=\t \theta_0<\cdots<\t \theta_{\t M}=1$ and 
$\t c_1\geq\cdots\geq \t c_M>0$ 
such that
$$
\forall c>\t c_m,\quad
\limsup_{t\to+\infty}\bigg(\sup_{x\geq  ct}v(t,x)\bigg)\leq \t \theta_{m-1},
\qquad
\forall c<\t c_m,\quad
\liminf_{t\to+\infty}\bigg(\inf_{x\leq  ct}v(t,x)\bigg)\geq \t \theta_m.
$$
Moreover, for any $m=1,\dots,\t M$, there exists a planar wave connecting 
$\t\theta_m$ to $\t\theta_{m-1}$ with speed $\t c_m$, that is, a strictly 
decreasing solution 
of the type $\phi(x-\t c_m t)$ with $\phi(-\infty)=\t\theta_m$,
$\phi(+\infty)=\t\theta_{m-1}$. We actually need an extension to more general
initial data, provided by \cite[Theorems~1.1, 2.14(v)]{Terracik}. Namely,
letting $S$ be from \eqref{f>0}, the result holds true for $v_0$ satisfying
\Fi{f-lcik}
v_0(-\infty)\in(S,1],\qquad v_0=0\ \text{ in }\R_+,\qquad v_0\ \text{ nonincreasing}.
\Ff
Then, the sequences $(\theta_m)_{m=1,\dots,M}$, $(c_m)_{m=1,\dots,M}$ in \thm{terrace} are obtained 
removing from $(\t\theta_m)_{m=1,\dots,\t M}$, $(\t c_m)_{m=1,\dots,\t M}$
all the elements $\t \theta_m$, $\t c_m$ such that $\t c_m=\t c_n$ for some $n>m$.
Notice that, in such way, $(c_m)_{m=1,\dots,M}$ is strictly decreasing and 
$v$ satisfies \eqref{c>cm}-\eqref{c<cm} for $x\geq0$.
%
%
%
%

Now, in order make the arguments of the proof of the Metatheorem work we just 
need the above $1$-dimensional spreading result to hold true with $1$ replaced by any 
of the levels $\theta_m$. Namely, we need the analogue of hypothesis
\eqref{DGM}.
\begin{proposition}\label{w}
For any $m\in\{1,\dots,M\}$, there exists a solution $w$ of \eqref{homo}
in dimension $N\!=\!1$, having a compactly supported, continuous
 initial datum $0\!\leq\!w_0\!<\!\theta_m$, such that $w(t,x)\to\theta_m$ as $t\to+\infty$.
\end{proposition}
The first ingredient to prove Proposition \ref{w} is the stability of the $\theta_m$ from below:
\Fi{thetam}
\forall m\in\{1,\dots,M\},\quad\exists S_m\in (\theta_{m-1},\theta_m),\quad
f>0\quad\text{in }\ (S_m,\theta_m).
  \Ff
This property in the case $m=M$ is just \eqref{f>0}, whereas the other
cases are provided by \cite[Lemma~4.3]{Terrace}.   
The second ingredient is the following ODE~result.

\begin{lemma}\label{lem:q}
Let $\phi$ be a bounded, strictly decreasing solution of
$\phi''+c\phi'+f(\phi)=0$ in $\R$, with $c>0$,
and let $q$ be a solution of $q''+f(q)=0$ in $\R$ satisfying $q(0)=\phi(0)$ and 
$q'(0)\leq\phi'(0)$ (resp. $\phi'(0)\leq q'(0)<0$). Then 
$$\inf_{\R_+}q<\phi(+\infty),\qquad 
\text{(resp. }\sup_{\R_-}q<\phi(-\infty)\text{)}.$$
\end{lemma}

\begin{proof}
Observe preliminarily that $\phi'$ is a nonpositive solution of $(\phi')''+c(\phi')'+f(\phi)'\phi'=0$,
thus it cannot vanish anywhere because otherwise it would be identically equal to $0$.
Direct computation reveals that the function $\Phi(u):=\phi'(\phi^{-1}(u))$	solves the equation
$\Phi'\Phi+c\Phi+f(u)=0$ for $u\in(\phi(-\infty),\phi(+\infty))$.

Let $q$ satisfy one of the two hypotheses of the lemma. In both case we have that $q'(0)<0$, because $\phi'(0)<0$.
Let $(x_1,x_2)$ be the largest interval  (possibly unbounded) containing $0$ in which $q'<0$. Then call $\eta:=\phi(0)=q(0)$ and 
$$\alpha:=\max\{q(x_2),\phi(+\infty)\}<\eta<\beta:=\min\{q(x_1),\phi(-\infty)\}.$$
The function $Q(u):=q'(q^{-1}(u))$ is a solution of $Q'Q+f(u)=0$ for $u\in(\alpha,\beta)$.
Subtracting the equations for $Q$ and $\Phi$ and integrating between $\eta$ and $u\in(\alpha,\beta)$~yields
\Fi{Q}
Q^2(u)-\Phi^2(u)=Q^2(\eta)-\Phi^2(\eta)+2c\int_{\eta}^u\Phi=
(q'(0))^2-(\phi'(0))^2+2c\int_{\eta}^u\Phi.
\Ff
Suppose that $q'(0)\leq\phi'(0)$. Then, by \eqref{Q},
$$\liminf_{u\to\alpha^+}Q^2(u)\geq 2c\int_{\eta}^\alpha\Phi>0.$$
If $q(x_2)\geq\phi(+\infty)$ then $\alpha=q(x_2)$ and the above inequality rewrites
$$\liminf_{x\to x_2}(q'(x))^2>0,$$
which contradicts the definition of $x_2$. As a consequence, 
$\inf_{\R_+}q\leq q(x_2)<\phi(+\infty)$.

Consider now the case $\phi'(0)\leq q'(0)<0$. Using again \eqref{Q} we get 
$$\liminf_{u\to\beta^-} \Phi^2(u)\geq\liminf_{u\to\beta^-}Q^2(u)-
2c\int_{\eta}^\beta\Phi\geq-2c\int_{\eta}^\beta\Phi>0.$$
If $q(x_1)\geq\phi(-\infty)$ then $\beta=\phi(-\infty)$ and therefore we have that
$$\liminf_{x\to -\infty}(\phi'(x))^2>0,$$
which is impossible because $\phi$ is bounded.
This shows that $q(x_1)<\phi(-\infty)$. If $x_1=-\infty$ we are done. Otherwise 
we necessarily have that $q'(x_1)=0$. It follows from the uniqueness for the Cauchy problem
that $q(x_1+\.)$ is even and then $q$ has a strict local maximum point at $x_1$.
If $q$ does not have other stationary points then $x_1$ is a global maximum point.
Otherwise $q$ is even with respect to the other stationary point too
and therefore it is periodic, whence $x_1$ is again a global maximum point.
We have thereby shown that $\sup_{\R_-}q=q(x_1)<\phi(-\infty)$.
\end{proof}

\begin{proof}[Proof of Proposition \ref{w}]
If $m=M$ the statement reduces to the hypothesis \eqref{DGM}.
Let $m<M$. Using the same notation as in the beginning of this section, we have that 
$\theta_m=\t \theta_j$ for some $j<\t M$.
Consider the profile $\phi$ of the planar wave connecting 
$\t\theta_j$ to $\t\theta_{j-1}$ translated in such a way that $\phi(0)\in
(S_m,\theta_m)$, where $S_m$ is from \eqref{thetam}.
Then let $q$ be the solution of $q''+f(q)=0$ in $\R$ with $q(0)=\phi(0)$
$q'(0)=\phi'(0)$. It follows from Lemma \ref{lem:q} that $\sup_{\R_-}q<\theta_m$ and 
$\inf_{\R_+}q<\t\theta_{j-1}$. From $\sup_{\R_-}q<\theta_m$ and $f>0$ in 
$(S_m,\theta_m)$ one readily sees that $q$ cannot be decreasing in the whole $\R_-$,
and then there exists $x_1<0$ such that $q'(x_1)=0$. Namely, $q(x_1+\.)$ is even.
In the case $j=1$, $\inf_{\R_+}q<\t\theta_{j-1}$ means that $q$ vanishes at some positive point. 
If $j>1$, it implies that $q$ intersects at some $x^j>0$ a suitable translation of the 
profile $\t\phi$ of the wave connecting $\t\theta_{j-1}$ to $\t\theta_{j-2}$ . 
Then, again by Lemma \ref{lem:q}, $\sup_{(x^j,+\infty)}q<\t\theta_{j-2}$.
We can repeat this argument until we find that $q$ vanishes at some positive point.
Let $x_2$ be the smallest of such points. Then define
$$
w_0(x)=\begin{cases}
q(x) & \text{for }x\in(2x_1-x_2,x_2)\\
0 & \text{otherwise}.
\end{cases}$$
This is a generalised subsolution of $q''+f(q)=0$. Hence, the solution $w$  of \eqref{homo} with $N=1$
emerging from $w_0$ is strictly increasing in $t$. As $t\to+\infty$, 
it converges locally uniformly to some limit $\ol w(x)$ satisfying
$w_0<\ol w\leq\theta_m$ and $\ol w''+f(\ol w)=0$. 
We show that $\ol w\equiv\theta_m$ using a variant of the sliding method.
First, there exists $h_0>0$ such that
$\ol w(x)\geq w_0(x\pm h)$ for all $h\in[0,h_0]$.
Then, by comparison, the evolutions by \eqref{homo} of the initial data $\ol w(x)$ and
$w_0(x\pm h)$ remain ordered, that is, $\ol w(x)\leq w(t,x\pm h)$. Letting $t\to+\infty$ we 
eventually derive $\ol w(x)\geq \ol w(x\pm h)$ for all $h\in[0,h_0]$, which
 means that $\ol w$ is constant. We deduce that $f(\ol w)=0$ from which, recalling that
$S_m<w_0(0)<\ol w\leq\theta_m$ and that $f>0$ in $(S_m,\theta_m)$,
we eventually get~$\ol w\equiv\theta_m$.
\end{proof}

\begin{proof}[Proof of \thm{terrace}]
Let $u$ be a bounded solution with a compactly supported
initial datum $0\leq u_0\leq 1$ such that $u(t,x)\to1$ as $t\to+\infty$
locally uniformly in $x\in\R^N$.

The upper bound \eqref{c>cm} is a direct consequence of the $1$-dimensional result of \cite{Terrace}
reclaimed at the beginning of the section. Namely, take $R>0$ such that $\supp u_0\subset B_R$. 
Then, the solution $v$ of \eqref{homo} in dimension 1 with initial
datum $\1_{(-\infty,R]}$ satisfies
$$
\forall c>c_m,\quad
\limsup_{t\to+\infty}\bigg(\sup_{|x|\geq  ct}v(t,|x|)\bigg)\leq \theta_{m-1}.
$$
Moreover, by comparison, $u(t,x)\leq v(t,x\.e)$ for all $e\in S^{N-1}$,
whence $u(t,x)\leq v(t,|x|)$. It follows that $u$ satisfies \eqref{c>cm}.

%
%
%
Let us prove the lower bound.
Assume by contradiction that \eqref{c<cm} does not hold for some $m$.
Then, proceeding exactly as in the proof of \thm{subset}, for any
$\eta<\theta_m$ close enough to $\theta_m$,  we find a
sequence of translations of $u$ converging to a function $u^*$ which is an 
entire solution of the same equation \eqref{homo} and satisfies in addition
$$u^*(0,0)=\eta,\qquad
\forall T\geq 0,\ x\.e<-cT, \quad u^*(-T,x)\geq\eta,$$
for some $e\in S^{N-1}$ and $c<c_m$. 
It is not restrictive to assume that $\eta>S_m$, where $S_m$ is given by \eqref{thetam}.
%
Define the family of functions $(u^*_T)_{T\geq0}$ by
$u^*_T(t,x):=u^*(t-T,x-cTe)$. They satisfy $u^*_T(0,x)\geq \eta$ for $x\.e<0$.
Consider a continuous, nonincreasing function 
$v_0:\R\to\R$ such that $v_0(-\infty)\in(S_m,\theta_m)$, $v_0=0$ in $\R_+$  and
$$\forall T\geq0,\ x\in\R^N,\quad
u^*_T(0,x)\geq v_0(x\.e).$$
Let $v$ be the solution of~\eqref{homo} in dimension $N=1$ with initial datum $v_0$.
From one hand, the comparison principle yields $v(t,x\.e)\leq u^*_T(t,x)$ for all $T\geq0$, $t>0$, $x\in\R^N$.
From the other, $v_0$ fulfils the hypothesis \eqref{f-lcik} 
with $(S,1]$ replaced by $(S_m,\theta_m)$. Thus, owing to Proposition \ref{w}, 
we are in the hypotheses of \cite[Theorems~1.1, 2.14(v)]{Terracik} and 
therefore we know that the solution $v$
converges to the minimal propagating terrace connecting $\theta_m$ to~$0$.
Thanks to \cite[Theorem 1.10(ii)]{Terrace}, this terrace has the same
levels~$\t \theta_j$ and speeds $\t c_j$ as the minimal terrace connecting $1$ to $0$ 
up to the level~$\theta_m$.
In particular, because $c<c_m$, we deduce that
$$\liminf_{t\to+\infty}v(t,ct)\geq\theta_m.$$
Consequently, observing that $u^*(0,0)=u^*_T(T,cTe)\geq v(T,cT)$, 
we eventually infer that $u^*(0,0)\geq\theta_m$,
which is a contradiction because $u^*(0,0)=\eta<\theta_m$.
\end{proof}


\appendix
 
\setcounter{section}{1}
\setcounter{theorem}{0}
\section*{Appendix}

\begin{proof}[Proof of Lemma~\ref{lem:comparison}]
We consider the perturbations $(\ol u^\e)_{\e>0}$ of $\ol u$ defined by 
$\ol u^\e(t,x):=\ol u(t,x)+\e$.
By hypothesis, for $\e>0$, there exists $T_\e\leq0$ such that
$\ol u^\e(t,x)>\ul u(t,x)$ for all $t\leq T_\e$, $x\in\R^N$.
Assume by contradiction that there is $\e_0>0$ such that
$$\forall\e\in(0,\e_0),\quad\exists t\in(T_\e,0],x\in\R^N,\quad \ol u^\e(t,x)
<\ul u(t,x),$$
otherwise the lemma is proved by letting $\e\to0$.
For $\e\in(0,\e_0)$, let $t_\e\in[T_\e,0)$ be the infimum of $t$ for which 
$u^\e(t,x)<\ul u(t,x)$ for 
some $x\in\R^N$. 
Thus, $\ol u^\e\geq\ul u$ if $t\leq t_\e$ and, by the uniform continuity of 
$\ul u$ and $\ol u$, $\inf_{x\in\R^N}(\ol u^\e-\ul u)(t_\e,x)=0$.
The hypotheses on $\ul u$ and $\ol u$ imply the existence of some 
$\rho_\e\in\R$ such that
\Fi{rho}
\inf_{x\.e=\rho_\e}(\ol u^\e-\ul u)(t_\e,x)=0.
\Ff
We distinguishing three possible situations.

{\em Case 1) \  $(\rho_\e)_{\e\in(0,\e_0)}$ is bounded.}\\
Let $(x_\e)_{\e\in(0,\e_0)}$ be such that 
$$x_\e\.e=\rho_\e,\qquad\ol u^\e(t_\e,x_\e)-\ul u(t_\e,x_\e)<\e.$$
The functions $\ul u^\e(\.+t_\e,\.+x_\e)$ and $\ol u(\.+t_\e,\.+x_\e)$ converge 
(up to 
subsequences) locally uniformly, as $\e\to0$, respectively to a subsolution 
$u_*$ and a supersolution $u^*$ of a limiting equation \eqref{le} on 
$\R_-\times\R^N$ satisfying 
$$\ol u^*(0,0)=\ul u^*(0,0),\qquad \forall 
t\leq0,\ x\in\R^N,\quad \ol u^*(t,x)\geq\ul u^*(t,x),$$
where the last inequality holds because $\ol u^\e\geq\ul u$ if $t\leq t_\e$.
The strong comparison principle then yields $\ol u^*=\ul u^*$ in 
$\R_-\times\R^N$. But the boundedness of $x_\e\.e=\rho_\e$ for 
$\e\in(0,\e_0)$ implies 
on one hand that $\liminf_{x\.e\to-\infty}\ol u^*(t,x)\geq1$ uniformly in 
$t\leq0$, 
by~\eqref{olu>}, and on the other that
$$\forall x\in\R^N,\quad\limsup_{t\to-\infty}\ul u^*(t,x)\leq0,$$
by \eqref{ulu<mono} or \eqref{ulu<others}. This case is thereby ruled out.

{\em Case 2) \  $\inf_{\e\in(0,\e_0)}\rho_\e=-\infty$.}\\
Let $S$ be from \eqref{f-gen}, and take $\e\in(0,\e_0)$ such that $-\rho_\e$ is 
large enough to have
$$\inf_{\su{t<0}{x\.e\leq\rho_\e+1}}\ol u(t,x)>S.$$
It follows from the second condition in \eqref{f-gen} that $\ol u^\e$ is a 
supersolution of \eqref{past-gen} for 
$x\in\O:=\{x\,:\,x\.e<\rho_\e+1\}$.
By \eqref{rho}, there is a sequence $\seq{y}$ in $\R^N$ 
such that
$$y_n\.e=0,\qquad \limn(\ol u^\e-\ul u)(t_\e,y_n+\rho_\e e)=0.$$
Passing to the limit on a subsequence of spatial 
translations by $(y_n)_{n\in\N}$ of $\ul u$ and $\ol u^\e$, we end up with a
subsolution $\ul u_\infty$ and a supersolution $\ol u^\e_\infty$ to some 
limiting equation~\eqref{le} in $\R_-\times\O$ satisfying
$$\ol u^\e_\infty(t_\e,\rho_\e e)=\ul u_\infty(t_\e,\rho_\e e),\qquad \forall 
t\leq t_\e,\ x\in\R^N,\quad
\ol u^\e_\infty(t,x)\geq\ul u_\infty(t,x).$$
It then follows from the strong comparison principle that $\ol
u^\e_\infty=\ul u_\infty$ for $t\leq t_\e$, $x\in\O$, which is
impossible because, by \eqref{olu>}, $\ol u^\e_\infty(t,x)>1$ if
$-x\.e$ is large enough, while $\ul u_\infty\leq1$.

{\em Case 3) \  $\sup_{\e\in(0,\e_0)}\rho_\e=+\infty$.}\\
This case is ruled out when $\ul u$ satisfies \eqref{ulu<mono}  because, in 
such case, \eqref{rho} yields $\rho_\e<\gamma t_\e+L<L$. Then, suppose that 
$f$ satisfies \eqref{others} and that $\ul u$ satisfies \eqref{ulu<others}.
By the latter, there is
$\e\in(0,\e_0)$ for which $\rho_\e$ is sufficiently 
large to have 
$$\ul u(t,x)\leq\theta\quad\text{for  }\ t\leq0,\ 
x\in\O:=\{x\,:\,x\.e>\rho_\e-1\},$$
where $\theta$ is from assumption \eqref{others}. It follows from that 
assumption that the function $\ul 
u^\e:=\ul u-\e$ is a subsolution of \eqref{past-gen} for $x\in\O$.
Moreover, $\ol u\geq\ul u^\e$ if $t\leq t_\e$ and, by~\eqref{rho},
$\inf_{x\.e=\rho_\e}(\ol u-\ul u_\e)(t_\e,x)=0$.
Arguing as in the case 2, one finds that the limits $\ol 
u_\infty$, $\ul u^\e_\infty$ of some sequences of 
translations of $\ol u$, $\ul u^\e$ by vectors orthogonal to $e$ coincide for
$t\leq t_\e$, which is impossible because $\ul u^\e<0$ if $x\.e$ is large 
enough.
\end{proof}


\end{document}